\documentclass[a4paper, 10pt, oneside, reqno]{amsart}

\usepackage[utf8]{inputenc}
\usepackage{amsfonts,amssymb,amsthm,amsmath}
\usepackage{mathrsfs}
\usepackage{geometry}
\usepackage{graphicx,graphics}
\usepackage{color}
\usepackage{tikz-cd}
\usepackage{comment}
\usepackage{subfiles}
\usepackage{stmaryrd}
\usepackage{hyperref}

\hypersetup{colorlinks = true, linkbordercolor = {white}, linkcolor = {blue}, citecolor = {blue}}

\theoremstyle{plain}
\newtheorem{thm}{Theorem}[section]
\newtheorem{lem}[thm]{Lemma}
\newtheorem{prop}[thm]{Proposition}
\newtheorem{corol}[thm]{Corollary}
\theoremstyle{definition}
\newtheorem{defi}[thm]{Definition}

\theoremstyle{remark}
\newtheorem{rque}[thm]{Remark}

\def \F {\mathbb{F}}

\def \Fq {\F_q}
\def \Fl {\F_{\ell}}

\def \Fqb {\F}
\def \Flb {\overline{\F}_{\ell}}

\def \Z {\mathbb{Z}}
\def \Zl {\Z_{\ell}}
\def \Zlb {\overline{\Z}_{\ell}}

\def \Q {\mathbb{Q}}
\def \Ql {\Q_{\ell}}
\def \Qlb {\overline{\Q}_{\ell}}

\def \Gm {\mathbb{G}_m}
\def \Ga {\mathbb{G}_a}

\def \A {\mathbb{A}}
\def \SL {\mathrm{SL}}

\def \id {\mathrm{id}}

\def \Tr {\mathrm{Tr}}

\def \fun {\mathrm{fun}}

\def \1 {\mathbb{1}}

\def \Lpsi {\mathcal{L}_{\psi}}

\def \LT {\mathrm{LT}}
\def \Frob {\mathrm{Frob}}

\def \Frob {\mathrm{F}}

\def \Av {\mathrm{Av}}

\def \Br {\mathrm{Br}}

\def \pr {\mathrm{pr}}
\def \RGamma {\mathrm{R\Gamma}}

\def \DD {\mathrm{D}}
\def \pt {\mathrm{pt}}

\title{Kazhdan-Laumon sheaves and Deligne-Lusztig representations}
\author{Arnaud Eteve}

\begin{document}

\maketitle

\noindent\textbf{Abstract.} Let $G$ be a reductive group over a finite field with a maximal unipotent subgroup $U$, we consider certain sheaves on $G/U$ defined by Kazhdan and Laumon and show that their cohomology produces the cohomology of the Deligne-Lusztig varieties. We then use this comparison to give a new proof of a result of Dudas. 

\bigskip

\tableofcontents

\section{Introduction}

Let $\Fq$ be a finite field of characteristic $p$ with algebraic closure $\Fqb$, and $G$ a reductive group over $\Fqb$ equipped with an endomorphism $\Frob_G : G \rightarrow G$ coming from an $\Fq$-structure. Let $B = TU$ be an $\Frob_G$-stable Borel pair and $W$ the corresponding Weyl group. In an attempt to understand the representations of the group $G^{\Frob_G}$, Deligne and Lusztig studied the $\ell$-adic cohomology of a family of varieties \cite{DeligneLusztig}, now called the Deligne-Lusztig varieties. Given a lifting $\dot{w}$ of $w$, we have a variety $Y(\dot{w})$ equipped with two commuting actions of $G^{\Frob_G}$ and $T^{w\Frob_T}$, we denote its cohomology by $R_w = \RGamma_c(Y(\dot{w}), \Qlb)$. 

In a later work, Kazhdan and Laumon built a family of perverse sheaves on $U \backslash G/U$ \cite{KazhdanLaumon}, which should encode an action of the Artin braid group $\Br_W$ of $W$ on the category of $\ell$-adic sheaves on $G/U$. They are parametrized by $w \in W$ and we denote them by $K_w^{\circ}$. They then used them to build certain representations of $G^{\Frob_G} \times T^{w\Frob_T}$, which we denote by $R_w'$, which are defined as the $K_0$ of certain categories of perverse sheaves. In an attempt to understand these representations, \cite{BravermanPolishchuk} computed the characters of $R'_{w,\theta}$ (the $\theta$-equivariant part of $R_w'$) when $\theta$ is in generic position and showed that they match with the characters of the Deligne-Lusztig representations $R_{w,\theta}$. 

The representations $R_{w, \theta}$ are in general not concentrated in a single degree while the construction of Kazhdan and Laumon produces a genuine vector space (and not a complex), as such we cannot expect them to match. To remedy this, one could try to consider instead of just the $K_0$, the whole $K$ theory, but this is in general too difficult to compute. An invariant that is simpler to compute is the categorical Hochschild homology, also called categorical trace, which has recently been used with great success in the geometric Langlands program \cite{6Authors}.

Roughly speaking, given a suitable $k$-linear category $\mathcal{C}$ and a suitable endofunctor $F$, one can define its trace $\Tr(F, \mathcal{C})$ which is now a (complex of) vector space functorial in $\mathcal{C}$ and $F$. A very interesting example is when $\mathcal{C} = \DD_c^b(X, \Qlb)$ is the category of $\ell$-adic sheaves on a scheme $X$, and $F = \Frob_{X,*}$. Then the trace is in general difficult to compute but we can approximate it by what it should be, that is $\fun(X^{\Frob_X},\Qlb)$ the space of functions on $X^{\Frob_X}$. In \cite{GaiVar}, the two authors build a morphism called the local term $\LT : \Tr(\Frob_{X,*}, \DD_c^b(X,\Qlb)) \rightarrow \fun(X^{\Frob_X},\Qlb)$ and show that it behaves functorially in $X$. 

More generally, given a sheaf $K$ on $X \times X$, we denote by $[K](-) = p_{2,!}(K \otimes p_{1}^*(-))$ the endofunctor of $\DD_c^b(X,\Lambda)$ defined by the kernel $K$. Then one expects a similar local term for the functor $[K] \circ \Frob_{X,*}$ to exist with target $\RGamma_c(X, (\id \times \Frob_X)^*K)$. Applying this to the setup of Kazhdan and Laumon, the action of $K_w^{\circ}$ can be realized as the action of the kernel $a^*K_w^{\circ}$ where $a : G/U \times G/U \rightarrow U \backslash G/U, a(g,h) = g^{-1}h$. Then one is enclined to look at the cohomology $\RGamma_c(G/U, (\id \times \Frob_{G/U})^*a^* K_w^{\circ})$, and we prove the following theorem which is valid for coefficients ring $\Lambda \in \{\Flb, \Zlb, \Qlb\}$. 

\begin{thm}[Theorem \ref{KLvsDL}]\label{KLvsDLexposition}
We have a $G^{\Frob_G} \times T^{w\Frob_T}$-equivariant isomorphism 
\begin{equation*}
\RGamma_c(G/U, (\id \times \Frob_{G/U})^*a^*K_w^{\circ}) \simeq \RGamma(Y(w), \Lambda)[\ell(w)].
\end{equation*}
\end{thm}

Our proof uses a geometric comparison between the varieties used to define the Kazhdan-Laumon sheaves and the Deligne-Lusztig varieties.

We fix an Artin-Schreier sheaf $\mathcal{L}_{\psi}$ on $\A^1$, we consider the following character of $\bar{U}$, the unipotent radical of the opposite Borel, 
\begin{equation*}
\chi : \bar{U} \rightarrow \bar{U}^{ab} \simeq \prod_{\alpha \in \Delta} \Ga \xrightarrow{\sum} \A^1
\end{equation*}
 and we still denote by $\mathcal{L}_{\psi}$ the pullback $\chi^*\mathcal{L}_{\psi}$. This sheaf also yields idempotent $e_{\psi} \in \Lambda[\bar{U}^{\Frob_{\bar{U}}}]$.

We then use \ref{KLvsDL} to give a new proof a theorem of Dudas.

\begin{thm}[\cite{Dudas}, see also Theorem \ref{thmDudas}]\label{thmDudasbis}
We have a $T^{w\Frob_T}$-linear isomorphism 
\begin{equation*}
e_{\psi} \RGamma(Y(w), \Lambda) = \Lambda[T^{w\Frob_T}][-\ell(w)].
\end{equation*}
\end{thm}

The idea of this proof comes again from the input of categorical traces. Namely the full subcategory $\DD_c^b((\bar{U},\psi) \backslash G /U, \Lambda)$ is stable under the action of $[K_w] \Frob_{G/U}$ and as such we can consider its trace, here the local term for this trace has target $\RGamma_c(G/U, (\id \times \Frob_{G/U})^*\Av_{\psi}K_w)$ where $\Av_{\psi}(-)$ is the functor $m_!(\mathcal{L}_{\psi}[2d_U] \boxtimes -)$ for $m : \bar{U} \times G/U \times G/U \rightarrow G/U \times G/U$ the action on the first factor. 

Now, by genericity of $\psi$, the category $\DD_c^b((\bar{U},\psi) \backslash G /U, \Lambda)$ is equivalent to the category $\DD_c^b(T, \Lambda)$ and, under this equivalence, the action of $K_w$ is entertwined with the action $w$ on $T$. This was already shown in \cite{Kazhdan1995} for the analog statement in the context of representations $p$-adic reductive groups and a proof in the $\ell$-adic setting is given in \cite{BezDesh}. 

We give a proof of a variant of this fact, following closely the argument of \cite{BezDesh}, by showing that the pullback to $T \times T$ of $\Av_{\psi}K_w$ is, up to shift, the constant sheaf on the graph of the action of $w$. Pulling back to the graph of Frobenius yields the constant sheaf on $T^{w\Frob_T}$ whose cohomology is nothing else than $\Lambda[T^{w\Frob_T}]$. The statement then follows from a comparison between $e_{\psi}\RGamma_c(G/U,  (\id \times \Frob_{G/U})^*K_w)$ and $\RGamma_c(G/U, (\id \times \Frob_{G/U})^*\Av_{\psi}K_w)$, see lemma \ref{lemAction1}, which we understand as comparison between geometric and arithmetic averaging. 

We plan to use the ideas presented here in a later work to explain a result of \cite{TZUJAN} on the endomorphisms of the Gelfand-Graev representations from a geometric point of view. 

\subsection*{Organization.} 

The paper is divided in three parts, in the first one we recall how the Kazhdan-Laumon sheaves are built. In the second one we prove \ref{KLvsDLexposition} and in the third section we give a new proof of \ref{thmDudasbis}.

\subsection*{Acknowledgments.} 

The author would like to thank Jean-François Dat for his continuous support and the careful reading of this text. We would like to thank Roman Bezrukavnikov who shared a preliminary version of \cite{BezDesh} with the author. We would also like to thank Thibault Alexandre, Sebastian Bartling and Thiago Landim for useful conversations and spotting some typos. 

\subsection*{Conventions.}

Let $\Fq$ be a finite field of characteristic $p$ with algebraic closure $\Fqb$, we consider schemes (and stacks) over $\Fqb$ but we always assume that they are pulled back from $\Fq$ and are equipped with a morphism $\Frob_X : X \rightarrow X$, we assume all morphism of schemes (and stack) to commute with $\Frob$. We fix $\Lambda$ a coefficient ring to be either $\Fl, \Zl, \Ql$ a finite extension of one of those or their algebraic closure. For a scheme (or more generally a stack) $X$ we denote by $\DD^b_c(X, \Lambda)$ the derived category of constructible sheaves on $X$ with coefficients in $\Lambda$, and by a sheaf on $X$ we always mean an element of $\DD^b_c(X, \Lambda)$. Given a morphism $f : X \rightarrow Y$, the notations $f_*,f^*,f_!$ and $f^!$ always mean the derived functors. Given $f : X \rightarrow Y$ a morphism and $A$ a sheaf on $Y$, we will sometimes use the notation $A(f) = f^*A$. 

Let $G$ be a reductive group (over $\Fqb$), we also fix a $\Frob_G$ stable Borel pair $B = TU$ and $W$ the corresponding Weyl group. We denote by $\bar{U}$ and $\bar{B}$ the opposite Borel and its unipotent radical. We denote by $\Delta$ the set of simple roots and given a simple root $\alpha$ we denote by $s_{\alpha} \in W$ the associated simple reflection. We denote by $G_{der}$ the derived subgroup of $G$ and by $G^{sc}$ the simply connected cover of $G_{der}$. The group $W$ acts on $T$, for $t \in T$ and $w \in W$ we denote by $w(t)$ the resulting action. Let $X$ be a scheme with an action of an algebraic group $G$ the quotient $X/G$ always means the stacky quotient, the categorical quotient will be denoted by $X\sslash G$.


\section{Kazhdan-Laumon sheaves}

In this section we recall a construction of Kazhdan and Laumon \cite{KazhdanLaumon} and extend their definition to the modular setting. 

\subsection*{A function on $\SL_2$.}

For $G = \SL_2$, we let $B = TU$ be the Borel of upper triangular matrices, $T$ the diagonal matrices, $\bar{B} = T\bar{U}$ be the opposite Borel. Let 
\begin{equation*}\alpha^{\vee} : \Gm \rightarrow T, z \mapsto 
\begin{pmatrix}
z & 0 \\
0 & \frac{1}{z}\\
\end{pmatrix},
\end{equation*} 
$s$ be the non trivial element of the Weyl group and  
\begin{equation*}
\dot{s} = 
\begin{pmatrix}
0 & -1 \\
1 & 0 \\
\end{pmatrix}.
\end{equation*}

Let $\phi : \SL_2 \rightarrow \A^1$ be the map 
$\begin{pmatrix}
a & b \\
c & d \\
\end{pmatrix} \mapsto c$.

\begin{lem}[\cite{BonnafeRouquier}]\label{lemPhiSL2}
For all $g \in \SL_2, t = \alpha(z) \in T$ and $u \in {U}$ we have 
\begin{enumerate}
\item $\phi({u}g) = \phi(g{u}) = \phi(g)$.
\item $\phi(tg) = z\phi(g)$ and $\phi(gt) = z^{-1}\phi(g)$.
\item $\phi(t^{-1}gs(t)) = \phi(g)$. 
\end{enumerate}
\end{lem}

\subsection*{Kazhdan-Laumon sheaf on $G$.} 

Assume that $G_{der}$ is simply connected. Let $\alpha \in \Delta$ be a simple root, and $s$ be the corresponding simple reflection, let $\xi_{\alpha} : \SL_2 \rightarrow G$ be the unique morphism extending $\alpha^{\vee}$ and sending $U_{\SL_2}$ to $U$ and denote by $G_{\alpha}$ its image. As in \cite{BonnafeRouquier} we let $\tau_{\alpha} : G_{\alpha}{U} \rightarrow G_{\alpha}$ be the projection and we let $\phi_{\alpha}$ be the composite

\begin{equation*}
G_{\alpha}{U} \xrightarrow{\tau_{\alpha}} G_{\alpha} \xrightarrow{\xi_{\alpha}^{-1}} \SL_2 \xrightarrow{\phi} \Ga.
\end{equation*}

Let $i_{\alpha} : G_{\alpha}{U} \rightarrow G$ be the inclusion, we then define 

\begin{equation*}
K_s^{\circ} = i_{\alpha,*}(-\phi_{\alpha})^*\mathcal{L}_{\psi}[2].
\end{equation*}

\begin{rque}
This is an object in $\DD^b_c(U \backslash G / U, \Lambda)$.
\end{rque}

\begin{defi}
We equip the category $\DD_c^b(U \backslash G / U)$ with the following convolution structure. Consider the diagram 

\[\begin{tikzcd}
	& {U\backslash G \times^UG/U} & {U\backslash G /U} \\
	{U\backslash G /U} & {} & {U\backslash G /U}
	\arrow["{\pr_1}", from=1-2, to=2-1]
	\arrow["{\pr_2}"', from=1-2, to=2-3]
	\arrow["m"', from=1-2, to=1-3]
\end{tikzcd}\]
where $U\backslash G \times^UG/U$ is the quotient of $U\backslash G \times G/U$ by the action of $U$ given by $u.(x,y) = (xu^{-1},uy)$. 

Then for $F,G \in \DD_c^b(U\backslash G /U, \Lambda)$ we set 
\begin{equation*}
F * G = m_!(\pr_1^*F \otimes \pr_2^*G).
\end{equation*}
\end{defi}

\begin{defi}
We equip the category $\DD_c^b(G/U, \Lambda)$ with a right action of $\DD_c^b(U \backslash G/U, \Lambda)$. Consider the following diagram 
\[\begin{tikzcd}
	& {G \times^UG/U} & {G /U} \\
	{G /U} & {} & {U\backslash G /U}
	\arrow["{\pr_1'}", from=1-2, to=2-1]
	\arrow["{\pr_2}"', from=1-2, to=2-3]
	\arrow["m"', from=1-2, to=1-3]
\end{tikzcd}\]
where the maps are defined as before. For $A \in \DD_c^b(G/U,\Lambda)$ and $B \in \DD_c^b(U \backslash G/U, \Lambda)$ we set 
\begin{equation*}
A * B = m_! (\pr_1'^*A \otimes \pr_2^*B).
\end{equation*}
\end{defi}

\begin{defi}
For a word $\underline{w} = s_1 \dots s_r$, we define $K_{\underline{w}}^{\circ} = K_{s_1}^{\circ} * \dots * K_{s_r}^{\circ}$.
\end{defi}

Let $a : G/U \times G/U \rightarrow U \backslash G / U$ be the map $a(g,h) = g^{-1}h$, this is the quotient map under the diagonal action of $G$ on $G/U \times G/U$. 

\begin{defi}
Given a constructible sheaf $K$ on $G/U \times G/U$ we define the functor $[K] : \DD_c^b(G/U, \Lambda) \rightarrow \DD_c^b(G/U, \Lambda)$ by 
\begin{equation*}
[K](-) = p_{2,!}(K \otimes p_1^*(-)),
\end{equation*}
where $p_i : G/U \times G/U \rightarrow G/U$ are the first and second projections.
\end{defi}

\begin{lem}\label{lemComparison}
Let $K$ be a sheaf on $U \backslash G/U$, we have a canonical isomorphism $[a^*K] \simeq (- * K)$ of functors $\DD_c^b(G/U, \Lambda) \rightarrow \DD_c^b(G/U, \Lambda)$. 
\end{lem}

\begin{proof}
Consider the following commutative diagrams
\[\begin{tikzcd}
	{G/U} & {G/U} \\
	{G \times^U G/U} & {G/U \times G/U} & {G \times^U G/U} & {G/U \times G/U} \\
	{G/U} & {G/U} & {U \backslash G/U} & {U \backslash G/U}
	\arrow["{(\pr_1', m)}", from=2-1, to=2-2]
	\arrow["{p_2}", from=2-2, to=3-2]
	\arrow[Rightarrow, no head, from=3-1, to=3-2]
	\arrow["m"', from=2-1, to=3-1]
	\arrow["{\pr_1'}", from=2-1, to=1-1]
	\arrow["{p_2}"', from=2-2, to=1-2]
	\arrow[Rightarrow, no head, from=1-1, to=1-2]
	\arrow["{(\pr_1', m)}", from=2-3, to=2-4]
	\arrow[Rightarrow, no head, from=3-3, to=3-4]
	\arrow["a", from=2-4, to=3-4]
	\arrow["{\pr_2}"', from=2-3, to=3-3]
\end{tikzcd}\]

We also consider the map $p_1 \times a' : G \times G/U \rightarrow G/U$ defined as $(g,h) \mapsto (g,g^{-1}h)$, it induces a map $G/U \times G/U \rightarrow G \times^U G/U$ which we denote by $p_1 \times a$. The map $p_1 \times a$ is easily seen to be the inverse map of $\pr_1' \times m$.

We can now compute 
\begin{align*}
- * K &= m_!(\pr_1'^* (-) \otimes \pr_2^*K) \\
&= p_{2,!}(\pr_1' \times m)_! (\pr_1^* (-) \otimes \pr_2^*K) \\
&= p_{2,!} (p_1 \times a)^* (\pr_1^* (-) \otimes \pr_2^*K) \\
&= p_{2,!}(p_1^*(-) \otimes a^* K) \\
&= [K](-).
\end{align*}
The second line comes from $(\pr_1 \times m)_! = (p_1 \times a)^*$ and the third from $\pr_1(p_1 \times a) = p_1$ and $\pr_2(p_1 \times a) = a$.
\end{proof}

\subsection*{The varities $\tilde{\mathcal{Y}}(\underline{w})$.}

Denote by $\mathcal{Y} = G/{U} \times G/U$. We keep the assumption that $G_{der}$ is simply connected, let $\underline{w} = s_1\dots s_r$ be an expression (not necessarily reduced).  

\begin{defi}
Let $s$ be a simple reflection and $\alpha$ the corresponding simple root and $G_{\alpha}$ the corresponding rank one subgroup (isomorphic to $\SL_2$). 
Let $\tilde{\mathcal{Y}}(s) \subset G/{U} \times G/{U}$ be the closed subscheme of pairs $(g{U}, h{U})$ such that $g^{-1}h \in G_{\alpha}{U}$. 

The map $\phi_{\alpha} : G_{\alpha}{U} \mapsto \Ga$ defines a map, still denoted by $\phi_{\alpha} : \tilde{\mathcal{Y}}(s) \rightarrow \Ga$. 
\end{defi}

\begin{defi}
Let $\tilde{\mathcal{Y}}(\underline{w}) = \tilde{\mathcal{Y}}(s_1) \times_{\mathcal{Y}} \dots \times_{\mathcal{Y}} \tilde{\mathcal{Y}}(s_r)$ be the closed subscheme of $(G/{U})^{r+1}$ of tuples $(g_1{U}, \dots, g_{r+1}{U})$ such that for all $1 \leq i \leq r, g_i^{-1}g_{i+1} \in G_{\alpha_{i}}{U}$.

\[\begin{tikzcd}
	{{\tilde{\mathcal{Y}}}(\underline{w})} & {\A^r} \\
	{{{\mathcal{Y}}}}
	\arrow["\phi_{\underline{w}}", from=1-1, to=1-2]
	\arrow["{p_{1,r+1}}"', from=1-1, to=2-1]
\end{tikzcd}\]

Denote by $p_{1,r+1}$ the projection on the outer copies of $G/U$ and by $\phi_{\underline{w}}$ the map

\begin{equation*}
(g_1{U}, \dots, g_{r+1}{U}) \mapsto (\phi_{s_1}(g_1^{-1}g_2), \dots, \phi_{s_r}(g_r^{-1}g_{r+1})).
\end{equation*}

We will drop the subscript $\underline{w}$ and write $\phi$ for $\phi_{\underline{w}}$ when it is clear from context.
\end{defi}

\begin{defi}
Let $\sigma : \A^r \rightarrow \Ga$ be the sum of the coordinates, define 
\begin{equation*}
K_{\underline{w}} = \phi^*\mathcal{L}_{\psi}(-\sigma)[2r]).
\end{equation*}
\end{defi}

\begin{prop}[\cite{KazhdanLaumon}]\label{propKL}
Suppose $\underline{w}$ is a reduced expression then the natural map 
\begin{equation*}
p_{1,r+1,!}K_{\underline{w}} \rightarrow p_{1,r+1,*}K_{\underline{w}}
\end{equation*}
is an isomorphism. 
\end{prop}

\begin{rque}
It follows from lemma \ref{lemComparison} that we have an isomorphism
\begin{equation*}
a^*K_{\underline{w}}^{\circ} = K_{\underline{w}}.
\end{equation*}
\end{rque}

\begin{rque}
The proof of \cite{KazhdanLaumon} is done for $\Qlb$ coefficient but these generalize to any coefficients as they appeal only to geometric properties of the varieties $\tilde{\mathcal{Y}}(\underline{w})$, properties of the Fourier-Deligne transform of \cite{LaumonFourier}, and theorem 4.2.5 of \cite{BBD} which all hold for sheaves with coefficients in $\Lambda$. 
\end{rque}

\begin{rque}
The image of $\dot{s} \in \SL_2$ under the map $\xi_{\alpha}$ defines an element $\dot{s}_{\alpha}$ which is a lift of the element $s_{\alpha} \in W$. Given a reduced expression $w = s_1 \dots s_r$ we denote by $\dot{w} = \dot{s}_1 \dots \dot{s}_r$, it is known that this element does not depend on the reduced expression of $w$ and is a lift of $w \in W$. We fix these lifts and the element $\dot{w}$ will always denote the one defined here. The sheaves $K_{\underline{w}}$ depend on these lifts. 
\end{rque}

\subsection*{General case}

Let $G$ be a reductive group and $G^{sc}$ be the simply connected cover of $G_{der}$ and $h : G^{sc} \rightarrow G$ be the corresponding map, we then define $K_w \in \DD_c^b(G, \Lambda)$ to be $h_!K_w$. Note that this definition is compatible with the definition of $K_w$ when $G$ has simply connected derived subgroup.

\section{Kazhdan-Laumon sheaves versus Deligne-Lusztig representations}

We fix $w \in W$, a reduced expression $\underline{w} = s_1\dots s_r$ and $r = \ell(w)$. 

\subsection*{Recollection on Deligne-Lusztig varieties.} 

We denote by
\begin{enumerate}
\item $\mathcal{X} = G/B \times G/B$,
\item $\mathcal{X}(w)$ the $G$-orbit corresponding to $w$,
\item $\overline{\mathcal{X}}(w)$ the closure pf $\mathcal{X}(w)$
\item and $q: \overline{\mathcal{X}}(\underline{w}) \rightarrow \overline{\mathcal{X}}(w)$ the Demazure resolution.
\end{enumerate}

Similarly for the Deligne-Lusztig varieties, consider the map $\id \times \Frob_{G/B} : G/B \rightarrow \mathcal{X}$, we denote by
\begin{enumerate}
\item $X(w) = G/B \times_{\mathcal{X}}\mathcal{X}(w)$ the Deligne-Lusztig variety,  
\item $\overline{X}(w) = G/B \times_{\mathcal{X}} \overline{\mathcal{X}}(\underline{w})$ its closure, 
\item and $\overline{X}(\underline{w}) = G/B \times_{\mathcal{X}}\overline{\mathcal{X}}(\underline{w})$ its resolution of singularities of \cite{DeligneLusztig}. 
\end{enumerate}

Let $Y(w) = Y(\dot{w})$ (recall that we have fixed lifts of $W$ in $N(T)$) be the $T^{w\Frob_T}$-covering of Deligne-Lusztig and let $\mathcal{F}$ be the lisse sheaf on $X(w)$ that is the pushforward of the constant sheaf on $Y(w)$.

\subsection*{More notations for Kazhdan-Laumon sheaves.}

Assume that $G$ has simply connected derived subgroup, we have already defined $\mathcal{Y} = G/U \times G/U$ and $\tilde{\mathcal{Y}}(\underline{w})$.  Recall that for $s$ a simple reflection we had a map $\phi_{\alpha_s} : \tilde{\mathcal{Y}}(s) \rightarrow \Ga$, we set $\tilde{\mathcal{Y}}_{\emptyset}(s)$ the inverse image of $\Gm$. We introduce the following notations:

\begin{enumerate}
\item $\tilde{\mathcal{Y}}_{\emptyset}(\underline{w}) = \tilde{\mathcal{Y}}_{\emptyset}(s_{1}) \times_{\tilde{\mathcal{Y}}} \tilde{\mathcal{Y}}_{\emptyset}(s_{2}) \times_{\tilde{\mathcal{Y}}} \dots \times_{\tilde{\mathcal{Y}}} \tilde{\mathcal{Y}}_{\emptyset}(s_{	r})$. Note that this is also the inverse image of $\Gm^r$ under $\phi : \tilde{\mathcal{Y}} \rightarrow \A^r$.
\item $\tilde{Y}_{\emptyset}(\underline{w}) = G/U \times_{\mathcal{Y}} \tilde{\mathcal{Y}}_{\emptyset}(\underline{w})$ where the map $G/U \rightarrow \mathcal{Y}$ is $\id \times \Frob_{G/U}$. 
\item ${\tilde{Y}}(\underline{w}) = G/U \times_{\mathcal{Y}} \tilde{\mathcal{Y}}(\underline{w})$.
\end{enumerate}

\subsection*{The simply connected case}

\begin{thm}\label{ThmKLtoDL}
Assume $G$ has simply connected derived subgroup. We have an isomorphism 

\begin{equation*}
\RGamma_c(G/U, (\id \times \Frob_{G/U})^* {K}_w) \simeq \RGamma(Y(w), \Lambda)[\ell(w)]. 
\end{equation*}

compatible with $G^{\Frob_G} \times T^{w\Frob_T}$-actions. 
\end{thm}

First note that the following diagram is Cartesian by definition
\[\begin{tikzcd}
	{{\tilde{Y}}(\underline{w})} & {{\tilde{\mathcal{Y}}}(\underline{w})} \\
	{{G/U}} & {{\mathcal{Y}}.}
	\arrow["{p_{1,r+1}}", from=1-2, to=2-2]
	\arrow["{p_{1,r+1}}"', from=1-1, to=2-1]
	\arrow["\id \times \Frob_{G/U}"', from=2-1, to=2-2]
	\arrow["i", from=1-1, to=1-2]
\end{tikzcd}\]

By proper base change, we have 
\begin{equation*}
\RGamma_c({\tilde{Y}}(\underline{w}), i^*\Lpsi(-\sigma \circ \phi)) = \RGamma_c(G/U, ((\id \times \Frob_{G/U})^* {K}_w).
\end{equation*}
It is then enough to show theorem \ref{ThmKLtoDL} for the left hand side. We show the stronger following property.

\begin{prop}\label{MainProp}
There is an isomorphism of sheaves on $\overline{X}(\underline{w})$ 
\begin{equation*}
p_!\phi^*\mathcal{L}_{\psi}(-\sigma)[2r] \simeq j_*\mathcal{F}[r]
\end{equation*}
where $j : X(w) \rightarrow \overline{X}(\underline{w})$ is the inclusion, $p : \tilde{Y}(\underline{w}) \rightarrow \overline{X}(\underline{w})$ is the projection and the map is induced by the adjunction map $\id \rightarrow j_*j^*$. 
\end{prop}

Proposition \ref{MainProp} follows from the two following lemmas. 

\begin{lem}\label{MainLem1}
The adjunction map 
\begin{equation*}
p_!\phi^*\mathcal{L}_{\psi}(-\sigma) \rightarrow j_*j^*p_!\phi^*\mathcal{L}_{\psi}(-\sigma)
\end{equation*}
is an isomorphism. 
\end{lem}

\begin{lem}\label{MainLem2}
There is an isomorphism compatible with the $G^{\Frob_G}$ and $T^{w\Frob_T}$ actions 
\begin{equation*}
j^*p_! \phi^*\mathcal{L}_{\psi}(-\sigma)[2r] \simeq \mathcal{F}[r].
\end{equation*}
\end{lem}

\subsection*{Proof of lemma \ref{MainLem1}} 

Consider the following diagram 

\[\begin{tikzcd}
	{\tilde{Y}(\underline{w})} & {\overline{X}(\underline{w}) \times_{\overline{\mathcal{X}}(\underline{w})} \tilde{\mathcal{Y}}(\underline{w})} & {\tilde{\mathcal{Y}}(\underline{w})} \\
	& {\overline{X}(\underline{w})} & {\overline{\mathcal{X}}(\underline{w})}
	\arrow["{i_X}"', from=2-2, to=2-3]
	\arrow["\pi", from=1-3, to=2-3]
	\arrow["{\pr_1}", from=1-2, to=2-2]
	\arrow["{\pr_2}", from=1-2, to=1-3]
	\arrow["{i_Y}", from=1-1, to=1-2]
	\arrow["p"', from=1-1, to=2-2]
\end{tikzcd}\]
where the square is Cartesian the map $i_X$ is the inclusion, the maps $\pr_i$ are the two projections and the map $i_Y$ is induced from the inclusion $\tilde{Y}(\underline{w}) \rightarrow \tilde{\mathcal{Y}}(\underline{w})$.

We first construct an action of $T \times \Gm^r$ on $\tilde{\mathcal{Y}}(\underline{w})$ as follow. Let $(t, z_1, \dots, z_r) \in T \times \Gm^r$. We define $t_i \in T$ by induction, by $t_1 = t$ and 
\begin{equation*}
t_{i+1} = s_i(t_i)\alpha^{\vee}_i(z_i).
\end{equation*}

The map $\nu : T \times \Gm^r \rightarrow T^{r+1}$ defined by $(t, z_i) \mapsto (t_i)$ is a inclusion. There is clearly an action of $T^{r+1}$ on $(G/U)^{r+1}$, as $\tilde{\mathcal{Y}}(\underline{w})$ is a closed subscheme of $(G/U)^{r+1}$ it is enough to see that it is stable under the action of $\nu(T \times \Gm^{r})$. We can decompose this in two steps.
\begin{enumerate}
\item If $t = 1$, then the action is the one defined in \cite{KazhdanLaumon}.
\item If for all $i, z_i = 1$, then given 
\begin{equation*}
(g_1U, \dots, g_{r+1}U) \in \tilde{\mathcal{Y}}(\underline{w}),
\end{equation*}
we want to see that $(g_1t_1U, \dots, g_{r+1}t_{r+1}U)$ is still in $\tilde{\mathcal{Y}}(\underline{w})$. By construction, for all $i$ we have $t_i = s_{i-1}\dots s_1(t_1)$, then it is enough to check that $t_i^{-1}g_i^{-1}g_{i+1}t_{i+1} \in G_{\alpha_i}U$. It is enough that $t_i^{-1}G_{\alpha_i}t_{i+1} = G_{\alpha_i}$ which is immediate.
\end{enumerate}

\begin{lem}\label{lemTorsor}
The map $\pi$ is $T \times \Gm^r$-equivariant for the trivial action of $T \times \Gm^r$ on the target and is a torsor under $T \times \Gm^r$.
\end{lem}

\begin{proof}
It is clear that the map $\tilde{\mathcal{Y}}(\underline{w}) \rightarrow \tilde{\mathcal{X}}(\underline{w})$ is equivariant for the trivial action of $T \times \Gm^r$ on $\overline{\mathcal{X}}(\underline{w})$. We compute $\tilde{\mathcal{Y}}(\underline{w}) \times_{\overline{\mathcal{X}}(\underline{w})} \tilde{\mathcal{Y}}(\underline{w})$. Let $(g_1U, \dots, g_{r+1}U)$ and $(g_1'U, \dots, g_{r+1}'U)$ be two elements of $\tilde{\mathcal{Y}}(\underline{w})$ such that $(g_1B, \dots, g_{r+1}B) = (g_1'B, \dots, g_{r+1}'B)$, then there exists a unique tuple $(t_1, \dots t_{r+1})$ such that $g_iU = g_i't_iU$ for all $i$. By definition of $\tilde{\mathcal{Y}}(\underline{w})$ the $t_i$ satisfy 
\begin{equation*}
s_i(t_i^{-1})t_{i+1} \in G_{\alpha_i}.
\end{equation*}
Hence there exists a unique $z_i \in \Gm$ such that 
\begin{equation*}
t_{i+1} = s_i(t_i)\alpha_i^{\vee}(z_i).
\end{equation*}
That is the tuple $(t_i)$ lives in $T \times \Gm^r$. 
\end{proof}

The second projection $\overline{X}(\underline{w}) \times_{\overline{\mathcal{X}}(\underline{w})} \tilde{\mathcal{Y}}(\underline{w}) \rightarrow \tilde{\mathcal{Y}}(\underline{w})$ is a closed embedding . An element 
\begin{equation*}
(g_1U, \dots, g_{r+1}U) \in \tilde{\mathcal{Y}}(\underline{w})
\end{equation*}
lies in $\overline{X}(\underline{w}) \times_{\mathcal{X}(\underline{w})} \tilde{\mathcal{Y}}(\underline{w})$ if an only if 
\begin{equation}
\Frob_G(g_1)B = g_{r+1}B.
\end{equation}

Hence there exists a unique $t \in T$ such that $\Frob_G(g_1)U = g_{r+1}tU$, this defines a map $\lambda : \overline{X}(\underline{w}) \times_{\overline{\mathcal{X}}(\underline{w})} \tilde{\mathcal{Y}}(\underline{w}) \rightarrow T$. We also denote by $\chi : \Gm^r \rightarrow T$ the following morphism 
\begin{equation*}
\Gm^r \rightarrow T \times \Gm^r \xrightarrow{\nu} T^{r+1} \xrightarrow{mult} T
\end{equation*}
where the first map is the inclusion and $mult$ is the multiplication.

\begin{lem}\label{lemTorsor2}
Consider the morphism $\rho : T \times \Gm^r \rightarrow T$ given by
\begin{equation*}
(t, z_i) \mapsto w^{-1}(\mathcal{L}_{w\Frob_T}(t)\chi(z_i^{-1})),
\end{equation*}
where $\mathcal{L}_{w\Frob_T}$ is the Lang map for the $w\Frob_T$-structure on $T$. Then the map $\lambda$ is equivariant for the action of $T \times \Gm^r$ on $T$ defined by $\rho$. 
\end{lem}

\begin{proof}
Let $(g_1U, \dots, g_{r+1}U) \in \overline{X}(\underline{w}) \times_{\overline{\mathcal{X}}(\underline{w})} \tilde{\mathcal{Y}}(\underline{w})$ and let $(t, z_i) \in T \times \Gm^r$ and $(t_i) \in T^{r+1}$ the corresponding elements. Let $\tilde{t} \in T$ be such that $\Frob_G(g_1)U = g_{r+1}\tilde{t}U$, solving the equation in $t' \in T$
\begin{equation*}
\Frob_G(g_1t_1)U = g_{r+1}t_{r+1}t'U
\end{equation*}
yields $t' = \Frob_T(t_1)t_{r+1}^{-1}\tilde{t}$. Therefore 
\begin{align*}
\lambda((g_iU).(t_i)) &= \lambda(g_iU)\Frob_T(t_1)t_{r+1}^{-1} \\
&= \lambda(g_iU)\Frob_T(t_1)w^{-1}(t_1 \chi(z_i^{-1})) \\
&= \lambda(g_iU)w^{-1}( \mathcal{L}_{w\Frob_T}(t) \chi(z_i^{-1})).
\end{align*} 
\end{proof}

\begin{lem}\label{lemTorsor3}
The map $\phi : \tilde{\mathcal{Y}}(\underline{w}) \rightarrow \A^r$ is equivariant for the action of $T \times \Gm^r$ on $\A^r$ given by the trivial action of $T$ and the natural action by dilatation of $\Gm^r$ on $\A^r$. 
\end{lem}

\begin{proof}
This is immediate by lemma \ref{lemPhiSL2} and the definition of the action of $T \times \Gm^r$.
\end{proof}

\begin{lem}\label{lemTorsor5}
The scheme $\tilde{Y}(\underline{w})$ is the fiber at $1$ of $\lambda$.
\end{lem}

\begin{proof}
The scheme $\tilde{Y}(\underline{w})$ is the closed subscheme of $\overline{X}(\underline{w}) \times_{\overline{\mathcal{X}}(\underline{w})} \tilde{\mathcal{Y}}(\underline{w})$ defined by the equation $\Frob_G(g_1)U = g_{r+1}U$, which is exactly the equation $\lambda = 1$. 
\end{proof}

\begin{lem}\label{lemTorsor4}
The map $\lambda \times \phi : \overline{X}(\underline{w}) \times_{\overline{\mathcal{X}}(\underline{w})} \tilde{\mathcal{Y}}(\underline{w}) \rightarrow T \times \A^r$ is smooth. 
\end{lem}

\begin{proof}
Let us first unfold the Lang map coming from the action of $T$, namely let $\tilde{Z}$ be the following pullback
\[\begin{tikzcd}
	{\tilde{Z}} & {\overline{X}(\underline{w}) \times_{\overline{\mathcal{X}}(\underline{w})} \tilde{\mathcal{Y}}(\underline{w})} \\
	{\A^r \times T} & {\A^r \times T.}
	\arrow["{\phi \times \lambda}", from=1-2, to=2-2]
	\arrow["{\id \times w^{-1}\mathcal{L}_{w\Frob_T}}"', from=2-1, to=2-2]
	\arrow[from=1-1, to=2-1]
	\arrow[from=1-1, to=1-2]
\end{tikzcd}\]
We have a closed immersion $\tilde{Z} \rightarrow \overline{X}(\underline{w}) \times_{\overline{\mathcal{X}}(\underline{w})} \tilde{\mathcal{Y}}(\underline{w}) \times T$. Consider now the diagonal action of $T$ on this scheme, the subscheme $\tilde{Z}$ is stable under this diagonal action. Let $\tilde{Z}_1$ be the fiber over $\A^r \times \{1\}$ of $\tilde{Z}$, then the action of $T$ yields an isomorphism $\tilde{Z}_1 \times T \rightarrow \tilde{Z}$ coming from the action of map of $T$ and fitting into a commutative diagram 
\[\begin{tikzcd}
	{\tilde{Z}_1 \times T} & {\tilde{Z}} \\
	{\A^r \times T} & {\A^r \times T}
	\arrow[from=2-1, to=2-2]
	\arrow[from=1-1, to=2-1]
	\arrow[from=1-1, to=1-2]
	\arrow[from=1-2, to=2-2]
\end{tikzcd}\]
where the map $\tilde{Z}_1 \times T \rightarrow \A^r \times T$ is given by the product of $\tilde{Z}_1 \rightarrow \A^r$ with the identity of $T$. Therefore it is enough to show that the map $\tilde{Z}_1 \rightarrow \A^r$ is smooth. Note that $\tilde{Z}_1 \rightarrow \tilde{Y}(\underline{w})$ is an isomorphism, hence it is enough to show that the map $\phi : \tilde{Y}(\underline{w}) \rightarrow \A^r$ is smooth. 

First note that the space $\tilde{\mathcal{Y}}(\underline{w})$ is smooth. This is by induction on $\ell(w)$. Consider $\underline{w}'$ the word where we have removed the first reflection.. Then we have a forgetful map $ \tilde{\mathcal{Y}}(\underline{w}) \rightarrow  \tilde{\mathcal{Y}}(\underline{w}')$ that forgets the first element. This is a $\SL_2/U$-fibration, which is a smooth. If $\underline{w}$ is empty then $\tilde{\mathcal{Y}}(\underline{w})$ is simply the diagonal of $G/U$ which is smooth. By \cite{DeligneLusztig} 9.11 and \cite{BonnafeRouquier} 2.5 the scheme $\tilde{Y}(\underline{w})$ is also smooth. Consider now $gU = (g_1U, \dots, g_{r+1}U) \in \tilde{Y}(\underline{w})$ we have an isomorphism of tangent spaces
\begin{equation*}
T_{gU} \tilde{\mathcal{Y}}(\underline{w}) \simeq T_{g_1^{-1}g_2} G_{\alpha_1}U/U \oplus \dots \oplus T_{g_r^{-1}g_{r+1}} G_{\alpha_r}U/U \oplus  T_{g_{r+1}} G/U,
\end{equation*}
such that the following diagram commutes 
\[\begin{tikzcd}
	{T_{gU} \tilde{\mathcal{Y}}(\underline{w})} & {} & {T_{g_1^{-1}g_2} G_{\alpha_1}U/U \oplus \dots \oplus T_{g_r^{-1}g_{r+1}} G_{\alpha_r}U/U \oplus  T_{g_{r+1}} G/U} \\
	& {T_{\phi(gU)}\A^r.}
	\arrow[from=1-1, to=1-3]
	\arrow["d\phi"', from=1-1, to=2-2]
	\arrow[from=1-3, to=2-2]
\end{tikzcd}\]
Here, the spaces $T_{gU} \tilde{\mathcal{Y}}(\underline{w})$ are the tangent spaces at the corresponding points and the map $T_{g_{r+1}} G/U \oplus T_{g_1^{-1}g_2} G_{\alpha_1}U/U \oplus \dots \oplus T_{g_r^{-1}g_{r+1}} G_{\alpha_1}U/U \rightarrow T_{\phi(gU)}\A^r$ is given by projection onto $T_{g_1^{-1}g_2} G_{\alpha_1}U/U \oplus \dots \oplus T_{g_r^{-1}g_{r+1}} G_{\alpha_r}U/U$ and then the product of the maps $d\phi_{\alpha} : T_{g'U}G_{\alpha} \rightarrow \A^1$. It is an immediate induction using the maps $\tilde{\mathcal{Y}}(\underline{w}) \rightarrow \tilde{\mathcal{Y}}(\underline{w}')$. The map $p_{1,r+1} : \tilde{\mathcal{Y}}(\underline{w}) \rightarrow G/U \times G/U$ has differential given in the previous coordinates by $dp_{1,r+1}(x, \xi_i) = (f(x,\xi_i),x)$ where $f$ is a linear map. 

The equation of $T_{gU}\tilde{Y}(\underline{w})$ in $T_{gU}\tilde{\mathcal{Y}}(\underline{w})$ is then given by $x = 0$ and the composite 
\begin{equation*}
T_{gU}\tilde{Y}(\underline{w}) \rightarrow T_{gU}\tilde{\mathcal{Y}}(\underline{w}) \rightarrow T_{g_1^{-1}g_2} G_{\alpha_1}U/U \oplus \dots \oplus T_{g_r^{-1}g_{r+1}} G_{\alpha_r}U/U
\end{equation*}
is surjective hence the map $d\phi : T_{gU}\tilde{Y}(\underline{w}) \rightarrow \A^r$ is surjective and the map $\phi$ is smooth. 
\end{proof}

We denote by $\widetilde{\Gm^r}$ the kernel of the map $\rho$. By lemmas \ref{lemTorsor2} and \ref{lemTorsor4}, the subscheme $\tilde{Y}(\underline{w})$ is stable under $\widetilde{\Gm^r}$.

\begin{lem}\label{lemTorsor6}
\begin{enumerate}
\item The map $p : \tilde{Y}(\underline{w}) \rightarrow \overline{X}(\underline{w})$ is a $\widetilde{\Gm^r}$-torsor. 
\item The map $\widetilde{\Gm^r} \rightarrow T \times \Gm^r \rightarrow \Gm^r$ induced by the second projection is surjective with kernel isomorphic to $T^{w\Frob_T}$. 
\end{enumerate}
\end{lem}

\begin{proof}
For the first part, we compute $\tilde{Y}(\underline{w}) \times_{\overline{X}(\underline{w})} \tilde{Y}(\underline{w})$. Given $\underline{g}U$ and $\underline{h}U \in \tilde{Y}(\underline{w})$ such that $\underline{g}B = \underline{h}B$, there exists $(t,z) \in T \times \Gm^r$ such that $\underline{g}U.(t,z) = \underline{h}U$. In particular after applying $\lambda$ we get $\lambda(\underline{h}U) = \rho(t,z)\lambda(\underline{g}U)$. By lemma \ref{lemTorsor2}, and by lemma \ref{lemTorsor5} we have $\lambda(\underline{g}U) = \lambda(\underline{h}U) = 1$ which implies that $\rho(t,z) = 1$ and $(t,z) \in \widetilde{\Gm^r}$.

For the second part, the surjectivity is clear by the surjectivity of the Lang map. The kernel of the map $\widetilde{\Gm^r} \rightarrow \Gm^r$ is the kernel of $\rho$ restricted to $T \times \{1\}$, which is the kernel of $\mathcal{L}_{w\Frob_T}$.
\end{proof}

In particular the group $\widetilde{\Gm^r}$ sits in an extension 
\begin{equation*}
1 \rightarrow T^{w\Frob_T} \rightarrow \widetilde{\Gm^r} \xrightarrow{q} \Gm^r \rightarrow 1.
\end{equation*}

\begin{proof}[Proof of lemma \ref{MainLem1}]
Our strategy is the following, we first consider $p^*p_! \phi^*\mathcal{L}_{\psi}$ and show that the corresponding morphism $\id \rightarrow j_*j^*$ is an isomorphism. Since $p$ is surjective, it is conservative, which will imply the claim on $\overline{X}(\underline{w})$. 

Consider the following diagram
\[\begin{tikzcd}
	& {\A^r \times \widetilde{\Gm^r}} & {\A^r} \\
	{\A^r \times \widetilde{\Gm^r}} & {\tilde{{Y}}(\underline{w}) \times \widetilde{\Gm^r}} & {\tilde{{Y}}(\underline{w})} \\
	{\A^r} & {\tilde{{Y}}(\underline{w})} & {\overline{X}(\underline{w})} \\
	{}
	\arrow["{q_1}", from=2-2, to=2-3]
	\arrow["{p}", from=2-3, to=3-3]
	\arrow["a"', from=2-2, to=3-2]
	\arrow["{p}"', from=3-2, to=3-3]
	\arrow["{\phi}", from=3-2, to=3-1]
	\arrow["{\phi \times \id}"', from=2-2, to=2-1]
	\arrow["a"', from=2-1, to=3-1]
	\arrow["{q_1}", from=1-2, to=1-3]
	\arrow["{\phi}"', from=2-3, to=1-3]
	\arrow["{\phi \times \id}", from=2-2, to=1-2]
\end{tikzcd}\]
where the maps $q_1$ are the first projections and the maps $a$ are the actions maps, for the action of $\widetilde{\Gm^r}$ on $\tilde{{Y}}(\underline{w})$ defined previously and the action of $\widetilde{\Gm^r}$ on $\A^r$ given by the projection $\widetilde{\Gm^r}\rightarrow \Gm^r$ and the dilatation action of $\Gm^r$ on $\A^r$. By the lemmas \ref{lemTorsor2} and \ref{lemTorsor3} diagram is commutative. The bottom right square is Cartesian by lemma \ref{lemTorsor6}, the two other squares are Cartesian.

For all the schemes in the above diagram we denote by $j$ the open inclusions
\begin{enumerate}
\item $X(\underline{w}) \subset \overline{X}(\underline{w})$,
\item $\tilde{Y}_{\emptyset}(\underline{w}) \subset \tilde{Y}(\underline{w})$, 
\item $\Gm^r \subset \A^r$,
\item $\tilde{Y}_{\emptyset}(\underline{w}) \times \widetilde{\Gm^r} \subset \tilde{Y}(\underline{w}) \times \widetilde{\Gm^r}$,
\item $\Gm^r \times \widetilde{\Gm^r} \subset\A^r \times \widetilde{\Gm^r}$.
\end{enumerate}

We now have a commutative diagram of functors $\DD_c^b(\A^r, \Lambda) \rightarrow \DD_c^b(\tilde{Y}(\underline{w}), \Lambda)$. 

\[\begin{tikzcd}
	{p^*p_!\phi^*} & {j_*j^*p^*p_!\phi^*} \\
	{q_{1,!}a^*\phi^*} & {j_*j^*q_{1,!}a^*\phi^*} \\
	{q_{1,!}(\phi \times \id)^*a^*} & {j_*j^*q_{1,!}(\phi \times \id)^*a^*} \\
	{\phi^*q_{1,!}a^*} & {j_*j^*\phi^*q_{1,!}a^*} \\
	{\phi^*q_{1,!}a^*} & {\phi^*j_*j^*q_{1,!}a^*}
	\arrow[from=1-1, to=1-2]
	\arrow[from=1-2, to=2-2]
	\arrow[from=1-1, to=2-1]
	\arrow[from=2-1, to=2-2]
	\arrow[from=2-1, to=3-1]
	\arrow[from=2-2, to=3-2]
	\arrow[from=3-1, to=3-2]
	\arrow[from=3-2, to=4-2]
	\arrow[from=3-1, to=4-1]
	\arrow[from=4-1, to=4-2]
	\arrow[Rightarrow, no head, from=4-1, to=5-1]
	\arrow[from=4-2, to=5-2]
	\arrow[from=5-1, to=5-2]
\end{tikzcd}\]
All horizontal maps are induced by the natural transformation $\id \rightarrow j_*j^*$ and vertical maps are either proper base change maps or induced the commutativity of the above diagram, the bottom right vertical map comes from smooth base change using that the map $\phi$ is smooth by \ref{lemTorsor4}. All vertical maps are isomorphisms therefore to show that the top horizontal map is an isomorphism, it is enough to show that the map $q_{1,!}a^* \rightarrow j_*j^*q_{1,!}a^*$ is an isomorphism. We only need to case of the sheaf $\mathcal{L}_{\psi}(-\sigma)$. 

Let us now unfold the definition of $q_{1,!}a^*\mathcal{L}_{\psi}(-\sigma)$. Consider the commutative diagram 
\[\begin{tikzcd}
	{\widetilde{\Gm^r}} & {\Gm^r} \\
	{\A^r \times \widetilde{\Gm^r}} & {\A^r \times \Gm^r} & {\A^r} \\
	{\A^r} & {\A^r} & {\A^r \times \A^r}
	\arrow["{\id \times q}", from=2-1, to=2-2]
	\arrow[Rightarrow, no head, from=3-1, to=3-2]
	\arrow["{p_1}"', from=2-1, to=3-1]
	\arrow["{p_1}", from=2-2, to=3-2]
	\arrow["a", from=2-2, to=2-3]
	\arrow["m"', from=3-3, to=2-3]
	\arrow["{\id \times j}"{description}, from=2-2, to=3-3]
	\arrow["{q_1}", from=3-3, to=3-2]
	\arrow["{q_2}"', from=2-2, to=1-2]
	\arrow[from=2-1, to=1-1]
	\arrow["q", from=1-1, to=1-2]
\end{tikzcd}\]
where $m$ is the multiplication $((x_i), (y_i)) \mapsto (x_iy_i), q_1$ is the projection on the first factor, $q_2$ the projection on the second factor and $a$ is again the action by dilation. We then have 

\begin{align*}
q_{1,!}a^*\mathcal{L}_{\psi}(-\sigma) &= q_{1,!}(q_2^*(q_!\Lambda) \otimes a^*\mathcal{L}_{\psi}(-\sigma)) \\
&= q_{1,!}(\id \times j)_!(q_2^*(q_!\Lambda) \otimes (\id \times j)^*m^*\mathcal{L}_{\psi}(-\sigma)) \\
&= q_{1,!}((q_2^*j_!q_!\Lambda) \otimes m^*\mathcal{L}_{\psi}(-\sigma)). 
\end{align*}

We introduce the Fourier-Deligne transform $\mathcal{F}_{-\psi} : \DD_c^b(\A^r, \Lambda) \rightarrow \DD_c^b(\A^r, \Lambda)$ of \cite{LaumonFourier}. Consider the diagram 
\[\begin{tikzcd}
	& {\A^r \times \A^r} & \A \\
	{\A^r} && {\A^r}
	\arrow["{\sigma \circ m}", from=1-2, to=1-3]
	\arrow["{q_1}"', from=1-2, to=2-3]
	\arrow["{q_2}", from=1-2, to=2-1]
\end{tikzcd}\]
then $\mathcal{F}_{-\psi}(A) = q_{1,!}(q^*_2 A \otimes (\sigma \circ m)^*\mathcal{L}_{\psi})[r]$. We then clearly have 
\begin{equation*}
q_{1,!}a^*\mathcal{L}_{\psi}(-\sigma) = \mathcal{F}_{-\psi}(j_!q_!\Lambda)[-r].
\end{equation*}
The sheaf $j_!q_!\Lambda$ has a finite filtration in $\DD_c^b(\A^r, \Lambda)$ with graded pieces given by $j_!\mathcal{L}_{\chi}$ where $\chi$ is a character of $T^{w\Frob_T}$ and $\mathcal{L}_{\chi}$ is the corresponding sheaf on $\Gm^r$. It is enough to show that $\id \rightarrow j_*j^*$ is an isomorphism for $\mathcal{F}_{-\psi}(j_!\mathcal{L}_{\chi})$ which is the content of lemma \ref{lemFourier}. 
\end{proof}

\begin{lem}\label{lemFourier}
Let $\mathcal{L}_{\theta}$ be a Kummer sheaf on $\Gm^r$ and $j : \Gm^r \rightarrow \A^r$ be the inclusion. As in the proof of lemma \ref{MainLem1}, we denote by $\mathcal{F}_{\psi}$ the Fourier transform functor. We have
\begin{equation*}
\mathcal{F}_{\psi}(j_!\mathcal{L}_{\theta}[r]) = j_*\mathcal{L}_{\theta^{-1}}[r] 
\end{equation*}
\end{lem}

\begin{proof}
We can decompose $\mathcal{L}_{\theta}$ as a tensor product $\boxtimes \mathcal{L}_{\theta_i}$ where $\boxtimes \mathcal{L}_{\theta_i}$ is a Kummer sheaf on $\Gm$. Then we have 
\begin{equation*}
\mathcal{F}(j_!\mathcal{L}_{\theta}) = \boxtimes \mathcal{F}(j_!\mathcal{L}_{\theta_i}).
\end{equation*}
which reduces the problem to a question in dimension one. Consider now the situation where $\mathcal{L}_{\theta}$ is a Kummer sheaf on $\Gm$ there are two cases.
\begin{enumerate}
\item The character $\theta$ is nontrivial, then by \cite{LaumonFourier} 1.4.3.2 we have $\mathcal{F}(j_!\mathcal{L}_{\theta}[1]) = j_*\mathcal{L}_{\theta^{-1}}[1]$ up to twist by a geometrically constant sheaf.
\item The character $\theta$ is trivial, then we have a triangle $\mathcal{F}(j_!\Lambda[1]) \rightarrow \mathcal{F}(\Lambda[1]) \rightarrow \mathcal{F}(i_*\Lambda[1])$ where $i : * \rightarrow \A^1$ is the origin. We have identifications $\mathcal{F}(\Lambda[1]) = i_*\Lambda$ and $\mathcal{F}(i_*\Lambda[1]) = \Lambda[1]$, and by \cite{KW} III.13.6, we can identify the Fourier transform of the adjunction map $\id \rightarrow i_*i^*$ with the adjunction map $i_!i^! \rightarrow \id$ which then yields $\mathcal{F}(j_!\Lambda[1]) = j_*\Lambda[1]$.
\end{enumerate}
\end{proof}

\subsection*{Proof of lemma \ref{MainLem2}}

We first consider the following diagram 
\[\begin{tikzcd}
	{Y(w)} & {\tilde{Y}_{\emptyset}(\underline{w})} & {\Gm^r} \\
	& {X(w)}
	\arrow["p", from=1-2, to=2-2]
	\arrow[from=1-1, to=1-2]
	\arrow[from=1-1, to=2-2]
	\arrow["\phi", from=1-2, to=1-3]
\end{tikzcd}\]
where $Y(w)$ is the Deligne-Lusztig variety which can be identified with the fiber at $(1, \dots, 1)$ of $\phi$. The inclusion $Y(w) \rightarrow \tilde{Y}_{\emptyset}(\underline{w})$ is also clearly equivariant for $T^{w\Frob_T} \subset \widetilde{\Gm^r}$. Since $Y(w) \rightarrow X(w)$ is a $T^{w\Frob_T}$-torsor, the action of $\widetilde{\Gm^r}$ induces an isomorphism 
\begin{equation*}
Y(w) \times^{T^{w\Frob_T}} \widetilde{\Gm^r} \simeq \tilde{Y}_{\emptyset}(\underline{w}), 
\end{equation*}
where $Y(w) \times^{T^{w\Frob_T}} \widetilde{\Gm^r}$ is the quotient of $Y(w) \times \widetilde{\Gm^r}$ by the action of $T^{w\Frob_T}$ given by $t.(x,y) = (t^{-1}x,ty)$. We can now compute 
\begin{align*}
p_!\phi^*\mathcal{L}_{\psi}(-\sigma)[2r] &= p_!(\Lambda_{Y(w)} \boxtimes^{T^{w\Frob_T}} q^*\mathcal{L}_{\psi}(-\sigma)[2r]) \\
&= (\mathcal{F} \otimes \RGamma_c(\widetilde{\Gm^r}, q^*\mathcal{L}_{\psi}(-\sigma)[2r]))^{T^{w\Frob_T}}.
\end{align*}

\begin{lem}\label{lemComputation}
There is a $T^{w\Frob_T}$-equivariant isomorphism 
\begin{equation*}
\RGamma_c(\widetilde{\Gm^r}, q^*\mathcal{L}_{\psi}(-\sigma)[2r]) = \Lambda[T^{w\Frob_T}][r]. 
\end{equation*}
\end{lem}

\begin{proof}
We first reduce the problem to a statement in rank one. By the Abhyankar lemma we can refine the covering $\widetilde{\Gm^r} \rightarrow \Gm^r$ into a covering $\Gm^r \rightarrow \widetilde{\Gm^r} \rightarrow \Gm^r$ such that the composition is given by $\alpha :  (x_i) \mapsto (x_i^{m_i})$ for some $m_i > 0$. Let $H$ be the kernel of $\Gm^r \rightarrow \widetilde{\Gm^r}$ then as $\Lambda[\prod_i \mu_i]^H = \Lambda[T^{w\Frob_T}]$ it is enough to that 

\begin{equation*}
\RGamma_c(\Gm^r, \alpha_*\Lambda \otimes \mathcal{L}_{\psi}(-\sigma)[2r]) = \Lambda[\prod_i \mu_i][r].
\end{equation*}
By the Kunneth formula 
\begin{equation*}
\RGamma_c(\Gm^r, \alpha_*\Lambda \otimes \mathcal{L}_{\psi}(-\sigma)[2r]) = \otimes_i \RGamma_c(\Gm, [m_i]_* \Lambda \otimes \mathcal{L}_{-\psi}[2]),
\end{equation*}
where $[m_i] : \Gm \rightarrow \Gm$ is the $m_i$th power map. We can then assume that $r = 1$. By further devissage we can assume that $m = m_1$ is a prime different from $p$. There are now two cases
\begin{enumerate}
\item $\ell$ is invertible in $\Lambda$ or $m \neq \ell$, 
\item $m = \ell$ and $\ell$ is not invertible in $\Lambda$.
\end{enumerate}

In the first case, $[m]_*\Lambda = \bigoplus \mathcal{L}_{\chi}$, where the sum is indexed by all characters of $\mu_m$. Then we have 
\begin{align*}
\RGamma_c(\Gm, \mathcal{L}_{-\psi} \otimes \mathcal{L}_{\chi})[2] = \Lambda_{\chi^{-1}}[1],
\end{align*}
where $\Lambda_{\chi}$ is the free $\Lambda$-module of rank one with action $\mu_m$ given by $\chi$.

For the second case, we introduce $\mathcal{C}$ the category of classical lisse sheaves of $\Lambda$-modules on $\Gm$ which are tame at $0$ and wild at $\infty$. This category was introduced by Katz in \cite{Katz}, on this category the functor $A \mapsto H^1_c(\Gm, A)$ is exact. In particular it induces an exact functor $\mathcal{C}^{T^{w\Frob_T}} \rightarrow \Lambda[T^{w\Frob_T}]$-mod. The sheaf $[\ell]_*\Lambda$ is an iterated extension of the constant sheaf, and as $\RGamma_c(\Gm, \mathcal{L}_{-\psi}[1]) = \Lambda$, we know that that $\RGamma_c(\Gm, [\ell]_* \Lambda \otimes \mathcal{L}_{-\psi}[1])$ is a free $\Lambda$-module of rank $\ell$. It remains to identify the $\mu_{\ell}$-structure. The ring $\Lambda[\mu_{\ell}]$ is a local ring and the reduction modulo the maximal ideal of $H^1_c(\Gm, [\ell]_*\Lambda \otimes \mathcal{L}_{\psi})$ is $H^1_c(\Gm, \Flb) = \Flb$. By Nakayama's lemma, $H^1_c(\Gm, [\ell]_*\Lambda \otimes \mathcal{L}_{\psi})$ is generated by one element and since it is of rank $\ell$ over $\Lambda$ it is isomorphic to $\Lambda[\mu_{\ell}]$. 
\end{proof}

\begin{proof}[Proof of lemma \ref{MainLem2}.] We can now compute
\begin{align*}
p_!\phi^*\mathcal{L}_{\psi}(-\sigma)[2r] &= (\mathcal{F} \otimes \RGamma_c(\widetilde{\Gm^r}, q^*\mathcal{L}_{\psi}(-\sigma)[2r]))^{T^{w\Frob_T}} \\
&\simeq (\mathcal{F} \otimes \Lambda[T^{w\Frob_T}])^{T^{w\Frob_T}}[r] \\
&\simeq \mathcal{F}[r].
\end{align*}
\end{proof}

\subsection*{The general case.}

In the previous section we assumed that $G$ had simply connected derived subgroup, we now extend the theorem to all $G$. Recall that we have a sheaf $K_w^{\circ} \in \DD_c^b(U \backslash G / U, \Lambda)$ for all $w$ and we also denoted the corresponding kernel $K_w \in \DD_c^b(G/U \times G/U, \Lambda)$ that is $K_w = a^* K_w^{\circ}$. 

\begin{thm}\label{KLvsDL}
We have a $G^{\Frob_G} \times T^{w\Frob_T}$-equivariant isomorphism 
\begin{equation*}
\RGamma_c(G/U, (\id \times \Frob_{G/U})^*K_w) \simeq \RGamma(Y(w), \Lambda)[\ell(w)].
\end{equation*}
\end{thm}

We already know that this is true if $G$ has simply connected derived subgroup. We fix $\pi : \tilde{G} \rightarrow G$ a central extension by a connected torus $Z$ such that $\tilde{G}$ has simply connected derived subgroup, we now do the reduction from $G$ to $\tilde{G}$. In particular the morphism $\pi$ induces an isomorphism on the unipotent radical of the corresponding Borel and on the Weyl groups, moreover the action of $W$ on $Z$ is trivial. We call $\tilde{B} = \tilde{T}U$ a Borel pair of $\tilde{G}$ that is sent to the Borel pair $B = TU$ of $G$. 

\begin{lem}
Let $w \in W$, then the morphism $\tilde{T}^{w\Frob_{\tilde{T}}} \rightarrow T^{w\Frob_T}$ is surjective with kernel $Z^{\Frob_Z}$. 
\end{lem}

\begin{proof}
We have a short exact sequence $1 \rightarrow Z \rightarrow \tilde{T} \rightarrow T \rightarrow 1$ applying the functor of $w\Frob_{\tilde{T}}$ invariant produces a short exact sequence $1 \rightarrow Z^{\Frob_Z} \rightarrow \tilde{T}^{w\Frob_{\tilde{T}}} \rightarrow T^{w\Frob_T} \rightarrow 1$ because $H^1(\Fq, Z) = 0$ by Lang's theorem since $Z$ is connected. 
\end{proof}

We denote by $\tilde{Y}(w)$ the Deligne-Lusztig variety associated to $\tilde{G}$ and $w$. 

\begin{lem}
The projection $\tilde{G}/U \rightarrow G/U$ induces an isomorphism 
\begin{equation*}
\tilde{Y}(w)/Z^{\Frob_Z} \simeq Y(w).
\end{equation*}
\end{lem}

\begin{proof}
Consider the following diagram 

\[\begin{tikzcd}
	{\tilde{G}/U} & {\tilde{Y}_Z(w)} & {\tilde{Y}(w)} \\
	{G/U} & {Y(w)}
	\arrow[from=1-2, to=1-1]
	\arrow[from=1-1, to=2-1]
	\arrow[from=2-2, to=2-1]
	\arrow[from=1-2, to=2-2]
	\arrow[from=1-3, to=1-2]
	\arrow[from=1-3, to=2-2]
\end{tikzcd}\]
where the first square is Cartesian and the morphisms $Y(w) \rightarrow G/U$ and $\tilde{Y}(w) \rightarrow \tilde{G}/U$ are the inclusions. The morphism $\tilde{Y}(w) \rightarrow \tilde{Y}_Z(w)$ is a closed immersion. For $\tilde{g}U \in \tilde{G}/U$ the equation defining $\tilde{Y}_Z(w)$ is $\mathcal{L}(g) \in UwU$ where $g$ is the image of $\tilde{g}$ in $G/U$. It is enough to see that $\tilde{Y}(w) \rightarrow Y(w)$ is a $Z^{\Frob_Z}$-torsor, as $\tilde{Y}_Z(w) \rightarrow Y(w)$ is a $Z$-torsor we can check it fiberwise. Let $gU \in Y(w)$ and let $\tilde{g}U \in \tilde{Y}_Z(w)$ be a lift of $gU$, the fiber in $\tilde{Y}_Z(w)$ is isomorphic to $Z$, as the Lang map of $Z$ is surjective we can assume that $\tilde{g}U$ is in $\tilde{Y}(w)$. Then in the fiber above $gU$, the equation defining $\tilde{Y}(w)$ is $\mathcal{L}_Z(z) = 1$ that is $z \in Z^{\Frob_Z}$.
\end{proof}

\begin{corol}
We have a canonical isomorphism 
\begin{equation*}
\RGamma(\tilde{Y}(w), \Lambda)^{Z^{\Frob_Z}} = \RGamma(Y(w), \Lambda).
\end{equation*}
\end{corol}

Consider the following diagram 

\[\begin{tikzcd}
	{\tilde{G}/U} & {\tilde{G}/UZ^{\Frob_Z}} & {G/U} \\
	{\tilde{G}/U \times \tilde{G}/U} & {\tilde{G}/U \times^Z \tilde{G}/U} & {G/U \times G/U} \\
	& {U \backslash \tilde{G}/U} & {U\backslash G/U}
	\arrow["{a_Z}", from=2-2, to=3-2]
	\arrow["{\tilde{a}}"', from=2-1, to=3-2]
	\arrow["{\pi_1}"', from=2-1, to=2-2]
	\arrow["{\pi_2}"', from=2-2, to=2-3]
	\arrow["a", from=2-3, to=3-3]
	\arrow[from=3-2, to=3-3]
	\arrow["{\id \times \Frob_{G/U}}", from=1-3, to=2-3]
	\arrow["{\id \times\Frob_{\tilde{G}/U}}"', from=1-1, to=2-1]
	\arrow["{\id \times^Z\Frob_{\tilde{G}/U}}"', from=1-2, to=2-2]
	\arrow["{q_1}", from=1-1, to=1-2]
	\arrow["{q_2}", from=1-2, to=1-3]
\end{tikzcd}\]

where, $\tilde{G}/U \times^Z \tilde{G}/U$ is the quotient by the diagonal action of $Z$, the morphisms $a, a_Z$ and $\tilde{a}$ are given by $a(g,h) = g^{-1}h$, the morphisms $\pi_i$ and $q_i$ are the projections. 

\begin{rque}
The two right squares are Cartesian. 
\end{rque}

Recall that the sheaf $K_w$ on $G/U \times G/U$ is defined as $a^* \pi_! \tilde{K}_w^{\circ}$ where $\tilde{K}_w^{\circ}$ is the corresponding Kazhdan-Laumon sheaf for $\tilde{G}$. 

\begin{lem}
We have a canonical isomorphism 
\begin{equation*}
\RGamma_c(G/U, (\id \times \Frob_{G/U})^* K_w) = \RGamma_c(\tilde{G}/U, (\id \times \Frob_{\tilde{G}/U})^* \tilde{K}_w)^{Z^{\Frob_Z}}.
\end{equation*}
\end{lem}

\begin{proof}
We first have 
\begin{align*}
(\id \times \Frob_{\tilde{G}/U})^* \tilde{K}_w &= (\id \times \Frob_{\tilde{G}/U})^* \tilde{a}^*\tilde{K}_w^{\circ} \\
&= q_1^* (\id \times^Z \Frob_{\tilde{G}/U})^* a_Z^* \tilde{K}_w^{\circ}. \\
\end{align*}

Then applying the functor $\RGamma_c(\tilde{G}/U, -)^{Z^{\Frob_Z}} = \RGamma_c(G/U, q_{2,!}q_{1,!} -)^{Z^{\Frob_Z}}$ we get

\begin{align*}
\RGamma_c(\tilde{G}/U, (\id \times \Frob_{\tilde{G}})^* \tilde{K}_w)^{Z^{\Frob_Z}} &= \RGamma_c(G/U, q_{2,!}q_{1,!}q_1^* (\id \times^Z \Frob_{\tilde{G}/U})^* a_Z^* \tilde{K}_w^{\circ})^{Z^{\Frob_Z}} \\
&= \RGamma_c(G/U, q_{2,!}(\id \times^Z \Frob_{\tilde{G}/U})^* a_Z^* \tilde{K}_w^{\circ}) \\
&= \RGamma_c(G/U, (\id \times \Frob_{G/U})^*\pi_{2,!}  a_Z^* \tilde{K}_w^{\circ})\\
&= \RGamma_c(G/U, (\id \times \Frob_{G/U})^*a^*\pi_! \tilde{K}_w^{\circ}) \\
&= \RGamma_c(G/U, (\id \times \Frob_{G/U})^*K_w).
\end{align*}
\end{proof}

\begin{proof}
The proof of theorem \ref{KLvsDL} is now straightforward we have isomorphisms, compatible with the $G^{\Frob_G} \times T^{w\Frob_T}$ actions giving 
\begin{align*}
\RGamma_c(G/U, (\id \times \Frob_{G/U})^* K_w) &= \RGamma_c(\tilde{G}/U, (\id \times \Frob_{\tilde{G}/U})^* \tilde{K}_w)^{Z^{\Frob_Z}}\\
&= \RGamma(\tilde{Y}(w), \Lambda)^{Z^{\Frob_Z}}[\ell(w)] \\
&= \RGamma(Y(w), \Lambda)[\ell(w)].
\end{align*}

\end{proof}

\section{A theorem of Dudas}

We now want to recover the following theorem of Dudas.

\begin{thm}[\cite{Dudas}]\label{thmDudas}
We have a $T^{w\Frob_T}$-linear isomorphism 
\begin{equation*}
e_{\psi} \RGamma(Y(w), \Lambda) = \Lambda[T^{w\Frob_T}][-\ell(w)].
\end{equation*}
\end{thm}

Firstly we consider the sheaf 

\begin{equation*}
\Av^1_{\psi}K_w
\end{equation*}
on $G/U \times G/U$ where $\Av^1_{\psi}$ is the averaging by $\bar{U}, \psi$ acting on the first coordinate, that is $\Av^1_{\psi}K_w := a_{1,!}(\mathcal{L}_{\psi}[2d_U] \boxtimes K_w)$ where $a_1 : \bar{U} \times G/U \times G/U \rightarrow G/U \times G/U$ is the multiplication on the first coordinate.

We will deduce the theorem out of the following lemmas.

\begin{lem}\label{lem1}
Assume that $G$ has a simply connected derived subgroup. Then the sheaf $\Av^1_{\psi}K_w$ is supported on $\bar{U}TU / U \times \bar{U}TU /U \simeq \bar{U} \times \bar{U} \times T \times T$ and is isomorphic to $\mathcal{L}_{\psi}[2d_U] \boxtimes \mathcal{L}_{\psi^{-1}} \boxtimes i_{w,*} \Lambda$ where $i_w : T \rightarrow T \times T$ is the map $t \mapsto (wt, t)$. 
\end{lem}

\begin{lem}\label{lem2}
We have a $U^{\Frob_U}$-equivariant isomorphism 
\begin{equation*}
\RGamma_c(G/U, (\id \times \Frob_{G/U})^* \Av^1_{\psi} K_w) \simeq e_{\psi}\RGamma_c(G/U, (\id \times \Frob_{G/U})^* K_w).
\end{equation*}
\end{lem}

\begin{proof}[Proof of theorem 	\ref{thmDudas}]

Firstly assume that $G$ has a simply connected derived subgroup and consider the following Cartesian diagram.

\[\begin{tikzcd}
	T & {T \times T} \\
	{T^{wF}} & T
	\arrow["{\id \times \Frob_T}"', from=2-2, to=1-2]
	\arrow["{i_w}", from=1-1, to=1-2]
	\arrow["i", from=2-1, to=1-1]
	\arrow["{i'}"', from=2-1, to=2-2]
\end{tikzcd}\]

From this we get that $\RGamma_c(T, (\id \times \Frob_T)^*i_{w,*}\Lambda) = \RGamma_c(T^{w\Frob_T}, \Lambda) = \Lambda[T^{w\Frob_T}]$. Theorem \ref{thmDudas} now follows from 

\begin{align*}
&e_{\psi} \RGamma(Y(w), \Lambda) \\
&= e_{\psi} \RGamma_c(G/U, (\id \times \Frob_{G/U})^*K_w)[-\ell(w)] \\
&= \RGamma_c(G/U, (\id \times \Frob_{G/U})^* \Av_{\psi} K_w)[-\ell(w)] \\
&= \RGamma_c(\bar{U} \times T , (\id \times \Frob_{\bar{U}})^* (\mathcal{L}_{\psi} \boxtimes \mathcal{L}_{\psi^{-1}}[2d_U]) \boxtimes (\id \times \Frob_{T})^* i_{w,*}\Lambda)[-\ell(w)] \\
&= \RGamma_c(\bar{U}, (\id \times \Frob_{\bar{U}})^*(\mathcal{L}_{\psi} \boxtimes \mathcal{L}_{\psi^{-1}}[2d_U])) \otimes \RGamma_c(T, \id \times \Frob_{T})^* i_{w,*}\Lambda)[-\ell(w)] \\
&= \Lambda[T^{w\Frob_T}][-\ell(w)].
\end{align*}

The first line is theorem \ref{KLvsDL}, the second one is lemma \ref{lem2}, the third one is lemma \ref{lem1} and the fourth one comes from the Kunneth formula and the identification $(\id \times \Frob_{\bar{U}})^*(\mathcal{L}_{\psi} \boxtimes \mathcal{L}_{\psi^{-1}}[2d_U])) \simeq \mathcal{L}_{\psi} \otimes \Frob_{\bar{U}}^*\mathcal{L}_{\psi^{-1}}[2d_U] \simeq \mathcal{L}_{\psi} \otimes \mathcal{L}_{\psi^{-1}}[2d_U] \simeq \Lambda[2d_U]$. For the general case, fix a central extension $\tilde{G} \rightarrow G$ by a connected torus such that $\tilde{G}$ has simply connected derived subgroup and then argue as in theorem \ref{KLvsDL}.
\end{proof}

\subsection*{Averaging $K_w$ by $\psi$.}

We now prove lemma \ref{lem1}.

\begin{lem}\label{lemConvol}
Let $X$ be a scheme and $K_1, K_2$ be two sheaves on $X \times X$. Recall that we have two functors $[K_i] : \DD_c^b(X, \Lambda) \rightarrow \DD_c^b(X, \Lambda)$. There is a canonical isomorphism of functors $[K_2] \circ [K_1] = [K_3]$ where $K_3 = \pr_{13,!}(\pr_{12}^* K_1 \otimes \pr_{23}^*K_2)$ and $\pr_{12},\pr_{23},\pr_{13} : X \times X \times X \rightarrow X \times X$ are the corresponding projections. We usually denote $K_3 = K_1 * K_2$.
\end{lem}

\begin{proof}
Consider the following diagram.

\[\begin{tikzcd}
	{X \times X \times X} & {X \times X} & X \\
	{X \times X} & X \\
	X
	\arrow["{\pr_{23}}", from=1-1, to=1-2]
	\arrow["{\pr_{12}}"', from=1-1, to=2-1]
	\arrow["{p_1}", from=1-2, to=2-2]
	\arrow["{q_2}"', from=2-1, to=2-2]
	\arrow["{p_2}", from=1-2, to=1-3]
	\arrow["{q_1}"', from=2-1, to=3-1]
\end{tikzcd}\]

The square is Cartesian, the maps $p_1,p_2,q_1,q_2$ are all projections. Let $A$ be a sheaf on $X$,

\begin{align*}
[K_2]\circ [K_1](A) &= p_{2,!}(K_2 \otimes p_1^* ([K_1]A)) \\
&= p_{2,!}(K_2 \otimes p_1^*q_{2,!}(K_1 \otimes q_1^* A)) \\
&= p_{2,!}(K_2 \otimes \pr_{23,!}\pr_{12}^*(K_1 \otimes q_1^* A)) \\
&= p_{2,!}\pr_{23,!} (\pr_{23}^* K_2 \otimes \pr_{12}^*K_1 \otimes \pr_{12}^*q_1^* A) \\
&= p_{2,!}(\pr_{13,!}(\pr_{23}^* K_2 \otimes \pr_{12}^*K_1) \otimes q_1^* A) \\
&= [K_1 * K_2](A).
\end{align*}

The first two lines are by definition, the third one is proper base change in the above diagram, the fourth one is the projection formula, the fifth one comes from the equalities $p_2\pr_{23} = p_2\pr_{13}$ and $q_1\pr_{12} = q_1\pr_{13}$ and the sixth one is again the projection formula. 
\end{proof}

\begin{lem}\label{lemConvol2}
Let $X$ be a scheme with an action of $U$ and let $K_1, K_2$ be two sheaves on $X \times X$ that are $U$-equivariant for the diagonal action of $U$. Then we have a canonical isomorphism 
\begin{equation*}
(\Av_{\psi}^1K_1) * (\Av_{\psi} K_2) = \Av_{\psi}^1 (K_1 * K_2).
\end{equation*}
\end{lem}

\begin{proof}
We introduce the functor $\Av_{\psi^{-1}}^2$ defined as
\begin{equation}
\Av_{\psi^{-1}}^2(-) = a_{2,!} (\mathcal{L}_{\psi^{-1}}[2d_{U}] \boxtimes -),
\end{equation}
where $a : U \times X \times X \rightarrow X \times X$ is the action on the second copy of $X$. 
Next consider the diagram 
\[\begin{tikzcd}
	{U \times X \times X} & {X \times X} \\
	{U \times X \times X \times X} & {X \times X \times X} & {X \times X} \\
	{U \times X \times X} & {X \times X.} \\
	{X \times X}
	\arrow["{\pr_{12}}", from=2-2, to=3-2]
	\arrow["\pr"', from=3-1, to=4-1]
	\arrow["{a_1}"', from=3-1, to=3-2]
	\arrow["{\id_U \times \pr_{12}}"', from=2-1, to=3-1]
	\arrow["{a_1 \times \id_{X \times X}}"', from=2-1, to=2-2]
	\arrow["{\pr_{23}}", from=2-2, to=2-3]
	\arrow["{\pr_{13}}"', from=2-2, to=1-2]
	\arrow["{a_1}", from=1-1, to=1-2]
	\arrow["{\id_U \times \pr_{13}}", from=2-1, to=1-1]
\end{tikzcd}\]

By proper base change we get an isomorphism $\Av_{\psi}^1(K_1 * K_2) = \Av_{\psi}^1(K_1) * K_2$ as follow
\begin{align*}
\Av_{\psi}^1(K_1) * K_2 &= \pr_{13,!}(\pr_{12}^*a_{1,!}(\mathcal{L}_{\psi}[2d_U] \boxtimes K_1) \otimes \pr_{23}^*K_2) \\
&= \pr_{13,!}((a_1 \times \id_{X \times X})_!(\id_U \times \pr_{12})^*(\mathcal{L}_{\psi}[2d_U] \boxtimes K_1) \otimes \pr_{23}^*K_2) \\
&= \pr_{13,!}(a_1 \times \id_{X \times X})_!((\id_U \times \pr_{12})^*(\mathcal{L}_{\psi}[2d_U] \boxtimes K_1) \otimes (a_1 \times \id_{X \times X})^*\pr_{23}^*K_2) \\
&= a_{1,!}(\id_U \times \pr_{13})_!((\id_U \times \pr_{12})^*(\mathcal{L}_{\psi}[2d_U] \boxtimes K_1) \otimes (a_1 \times \id_{X \times X})^*\pr_{23}^*K_2) \\
&= a_{1,!}(\id_U \times \pr_{13})_!((\mathcal{L}_{\psi}[2d_U]) \boxtimes (\pr_{12}^*K_1 \otimes \pr_{23}^*K_2)) \\
&= \Av_{\psi}^1(K_1 * K_2).
\end{align*}
Similarly we get a canonical isomorphism $\Av_{\psi^{-1}}^2(K_1 * K_2) = K_1 * \Av_{\psi^{-1}}^2K_2$. 

Now, for $K$ on $X \times X$ which is diagonally $U$-equivariant we have canonical isomorphisms
\begin{equation}
\Av_{\psi}^1K = \Av_{\psi^{-1}}^2\Av_{\psi}^1K = \Av_{\psi}^1\Av_{\psi^{-1}}^2K = \Av_{\psi^{-1}}^2K.
\end{equation}
We only construct the first two. For the first one consider the following diagram 
\[\begin{tikzcd}
	{U \times U \times X \times X} & {U \times U \times X \times X} & {U \times X \times X} \\
	& {X \times X }
	\arrow["{a_{12}}"', from=1-1, to=2-2]
	\arrow["{a_{1\Delta}}", from=1-2, to=2-2]
	\arrow["{\alpha\times\id_{X \times X}}", from=1-1, to=1-2]
	\arrow["{a_1}", from=1-3, to=2-2]
	\arrow["{\tilde{a}_{\Delta}}", from=1-2, to=1-3]
\end{tikzcd}\]
where $\alpha : U \times U \rightarrow U \times U$ is the map $(u,v) \mapsto (uv^{-1}, v)$, $a_{1\Delta}(u,v,x,y) = (uvx, vy)$ and $\tilde{a}_{\Delta}(u,v,x,y) = (u,vx,vy)$. Note that $\alpha_!(\mathcal{L}_{\psi}  \boxtimes \mathcal{L}_{\psi^{-1}}) = \mathcal{L}_{\psi}  \boxtimes \Lambda_U$. It then follows that 
\begin{align*}
a_{12,!}(\mathcal{L}_{\psi}  \boxtimes \mathcal{L}_{\psi^{-1}} \boxtimes K)[4d_U] &= a_{1\Delta,!}(\mathcal{L}_{\psi}  \boxtimes \Lambda_U \boxtimes K)[4d_U] \\
&= a_{1,!}\tilde{a}_{\Delta,!}(\mathcal{L}_{\psi} \boxtimes \Lambda_U \boxtimes K)[4d_U] \\
&= a_{1,!}(\mathcal{L}_{\psi} \boxtimes K)[2d_U].
\end{align*} 
In the last line, we have used that $K$ was diagonally $U$-equivariant. 

For the second isomorphism, consider 
\[\begin{tikzcd}
	{U \times U \times X \times X } & {U \times X \times X} \\
	{U \times X \times X} & {X \times X} \\
	&& {X \times X}
	\arrow["{\tilde{a}_1}", from=1-1, to=1-2]
	\arrow["{a_2}"{description}, from=1-2, to=2-2]
	\arrow["{\tilde{a}_2}"', from=1-1, to=2-1]
	\arrow["{a_1}"{description}, from=2-1, to=2-2]
	\arrow["\pr"', from=2-1, to=3-3]
	\arrow["\pr", from=1-2, to=3-3]
	\arrow["{a_{12}}"{description}, from=1-1, to=2-2]
\end{tikzcd}\]
where the square is Cartesian. We now have 
\begin{align*}
\Av_{\psi^{-1}}^2\Av_{\psi}^1(K) &= a_{2,!}(\mathcal{L}_{\psi^{-1}}[2d_U] \boxtimes \Av_{\psi}^1(K)) \\
&= a_{2,!}(\mathcal{L}_{\psi^{-1}}[2d_U] \boxtimes (a_{1,!}\mathcal{L}_{\psi}[2d_U] \boxtimes K)) \\
&= a_{12,!}( \mathcal{L}_{\psi}[2d_U] \boxtimes \mathcal{L}_{\psi^{-1}}[2d_U] \boxtimes K) \\
&= \Av_{\psi}^1\Av_{\psi^{-1}}^2(K)
\end{align*}

Noting that $K_1 * K_2$ is diagonally $U$-equivariant if both $K_1$ and $K_2$ are diagonally $U$-equivariant, we can combine the previous isomorphisms into
\begin{align*}
\Av_{\psi}^1(K_1 * K_2) &= \Av_{\psi^{-1}}^2\Av_{\psi}^1(K_1 * K_2) \\
&= (\Av_{\psi}^1K_1) * (\Av_{\psi^{-1}}^2K_2) \\
&= (\Av_{\psi}^1K_1) * (\Av_{\psi}^1K_2).
\end{align*}
\end{proof}

\begin{lem}\label{lemConvol3}
Let $w,w' \in W$. Then we have an isomorphism 
\begin{equation*}
i_{w,*}\Lambda * i_{w',*} \Lambda = i_{ww',*} \Lambda.
\end{equation*}
\end{lem}

\begin{proof}
We first note that $\pr_{12}^*i_{w,!}\Lambda = (i_w \times \id)_!\Lambda$ and  $\pr_{23}^*i_{w',!}\Lambda = (\id \times i_{w'})_!\Lambda$. Now we have $(i_w \times \id)_!\Lambda \otimes (\id \times i_{w'})_!\Lambda = (i_w \times \id)_!(i_w \times \id)^*(\id \times i_{w'})_!\Lambda$. We have a Cartesian diagram 

\[\begin{tikzcd}
	T & {T \times T} \\
	{T \times T} & {T \times T \times T}
	\arrow["{i_w \times\id}"', from=2-1, to=2-2]
	\arrow["{\id \times i_{w'}}", from=1-2, to=2-2]
	\arrow["{i_{ww'}}", from=1-1, to=1-2]
	\arrow["{i_{w'}}"', from=1-1, to=2-1]
\end{tikzcd}\]
giving by proper base change $(i_w \times \id)_!(i_w \times \id)^*(\id \times i_{w'})_!\Lambda = (i_w \times \id)_! i_{w',!}\Lambda$ finally we get that $\pr_{13,!}(\pr_{12}^*i_{w,!}\Lambda \otimes \pr_{23}^*i_{w',!}\Lambda) = \pr_{13,!}(i_w \times \id)_! i_{w',!}\Lambda = i_{ww',!}\Lambda$ since $\pr_{13}(i_w \times \id)i_{w'} = i_{ww'}$. 
\end{proof}

\begin{lem}\label{lemSupport}
Let $F$ be a sheaf on $G$ that is $(\bar{U}, \psi)$-equivariant on the left and $U$-equivariant on the right. Then $F$ is supported on $\bar{U}TU$. 
\end{lem}

\begin{proof}
The argument is well known see for instance \cite{BBM}. We recall it briefly, let $g \in G$ then $F_{|\bar{U}gU}$ is nonzero only if $\psi$ is trivial on $\bar{U} \cap (^gU)$, only if $\bar{U} \cap (^gU)$ does not contain a negative simple root subgroup, that is, when $g \in \bar{U}TU$. 
\end{proof}

The proof of the next lemma is due to Bezrukavnikov and Deshpande in \cite{BezDesh}, we reproduce the argument for the convenience of the reader. 

\begin{proof}[Proof of lemma \ref{lem1}]
The statement about the support follows from lemma \ref{lemSupport} for $G \times G$. The sheaf $\Av^1_{\psi}K_w$ is $\bar{U} \times \bar{U}, \psi \times \psi^{-1}$-equivariant, therefore on $\bar{U} \times \bar{U} \times T \times T$, it is of the form $\mathcal{L}_{\psi} [2d_U] \boxtimes \mathcal{L}_{\psi^{-1}} \boxtimes A$ where $A = i_{T \times T}^* \Av^1_{\psi}K_w [-2d_U]$ and $i_{T \times T}$ is the inclusion of $T \times T \rightarrow G \times G$. 

Then we can assume that $w$ is a simple reflection, by the definition of $K_w$ and lemmas \ref{lemConvol2} and \ref{lemConvol3}. This is now a direct computation, we therefore fix $s$ a simple reflection and $\alpha$ the associated simple root.

Consider the Cartesian diagram

\[\begin{tikzcd}
	{\bar{U} \times T \times T} & {\bar{U} \times G/U \times G/U} \\
	{T \times T} & {G/U \times G/U}
	\arrow["a", from=1-2, to=2-2]
	\arrow["{i_{T \times T}}"', from=2-1, to=2-2]
	\arrow["\lambda", from=1-1, to=1-2]
	\arrow["{\pr_2}"', from=1-1, to=2-1]
\end{tikzcd}\]
where $a$ is the action map on the first coordinate, $\pr_2$ is the projection on $T \times T$ and $\lambda(u,t,t') = (u,u^{-1}t,t')$. By proper base change we have 
\begin{equation*}
i_{T \times T}^* \Av^1_{\psi}K_s [-2d_U] = \pr_{2,!}\lambda^* (\mathcal{L}_{\psi} \boxtimes K_s).
\end{equation*}

The sheaf $\lambda^*K_s$ is supported on the triples $(u,t,t')$ such that $t^{-1}ut' \in G_{\alpha}U$, which amounts to $t^{-1}t' \in \alpha^{\vee}(\Gm) := T_s$ and $u \in \bar{U} \cap (^{t^{-1}}(G_{\alpha}U)) = U_{-\alpha}$. Therefore write $u = x_{-\alpha}(x)$ for $x \in \Ga$ and $t^{-1}t = \alpha^{\vee}(z)$. Consider the following diagram 
\[\begin{tikzcd}
	& {\Ga \times \Gm \times T} \\
	{\bar{U} \times T \times T} & {\bar{U}\times G/U \times G/U} & {G/U \times G/U} \\
	&& \Ga
	\arrow["{\phi_{\alpha}}", from=2-3, to=3-3]
	\arrow["{\pr_2}"', from=2-2, to=2-3]
	\arrow["\lambda"', from=2-1, to=2-2]
	\arrow["\delta", from=1-2, to=2-2]
	\arrow["{\delta'}"', from=1-2, to=2-1]
\end{tikzcd}\]
where $\delta(x,z,t) = (x_{-\alpha}(x), t, \alpha^{\vee}(z)t)$ and $\pr_2$ is the projection. The preceeding discussion can be summed up as the adjunction map $\lambda^*(\mathcal{L}_{\psi} \boxtimes K_s) \rightarrow \delta_*\delta^*\lambda^*(\mathcal{L}_{\psi} \boxtimes K_s)$ is an isomorphism. Consider the composition 
\begin{equation*}
\Ga \times \Gm \times T \xrightarrow{\delta} \bar{U} \times G/U \times G/U \xrightarrow{\pr_2} G/U \times G/U \xrightarrow{\phi_{\alpha}} \Ga.
\end{equation*}
We denote this map $\tau$, it is given by $\tau : (x,z,t) \mapsto (xz\alpha(t))$. The sheaf $\delta^*\lambda^*(\mathcal{L}_{\psi} \boxtimes K_s)$ is then given by $(\id_{\Ga} - \tau)^*\mathcal{L}_{\psi}[2]$. Consider the diagram 
\[\begin{tikzcd}
	& {\Ga \times \Gm \times T} \\
	{T \times T} & {\Ga \times \Ga} \\
	& \Ga
	\arrow["{\id_{\Ga} \times \tau}", from=1-2, to=2-2]
	\arrow["{\delta"}"', from=1-2, to=2-1]
	\arrow[from=2-2, to=3-2]
\end{tikzcd}\]
where $\delta"$ is the composition $\pr_2\delta'$ and the vertical map $\Ga \times \Ga \rightarrow \Ga$ is the map $(x,y) \mapsto x-y$. We have $\pr_{2,!}\lambda^*(\mathcal{L}_{\psi} \boxtimes K_s) = \delta_!"(\id_{\Ga} - \tau)^*\mathcal{L}_{\psi}[2]$. We first compute it fiberwise, let $(t,t') \in T \times T$. The fiber of $\delta"$ above $(t,t')$ is nonempty only if there exists $z \in \Gm$ such that $t' = \alpha^{\vee}(z)t$, in which case the fiber is isomorphic to $\Ga$ via the projection onto $\Ga$. Let $z \in \Gm$ and $t' = t\alpha^{\vee}(z)$, under the identification $\delta"^{-1}(t,t') \simeq \Ga$, the sheaf $(\id_{\Ga} - \tau)^*\mathcal{L}_{\psi}[2]_{|\delta"^{-1}(t,t')}$ corresponds to $m_{1 - z\alpha(t)}^*\mathcal{L}_{\psi}[2]$ where $m_x : \Ga \rightarrow \Ga$ is the multiplication by $x \in k$. It follows that the fiber of $\delta_!"(\id_{\Ga} - \tau)^*\mathcal{L}_{\psi}[2]$ is nonzero only if $z\alpha(t) = 1$, that is, when $z = \alpha(t)^{-1}$, which further reduces to $t' = s(t)$. Finally consider the Cartesian diagram
\[\begin{tikzcd}
	T & {\Ga \times T} \\
	{T \times T} & {\Ga \times \Gm \times T.}
	\arrow["{i_s}"', from=1-1, to=2-1]
	\arrow["{\tilde{\delta}}"', from=1-2, to=1-1]
	\arrow["{\delta"}", from=2-2, to=2-1]
	\arrow["{\tilde{i}_s}", from=1-2, to=2-2]
\end{tikzcd}\]
We have 
\begin{align*}
\delta_!"(\id_{\Ga} - \tau)^*\mathcal{L}_{\psi}[2] &= i_{s,!}i_s^*\delta_!"(\id_{\Ga} - \tau)^*\mathcal{L}_{\psi}[2] \\
&= i_{s,!}\tilde{\delta}_!\tilde{i_s}^*(\id_{\Ga} - \tau)^*\mathcal{L}_{\psi}[2] \\
&= i_{s,!}\tilde{\delta}_!\Lambda[2] \\
&= i_{s,!}\Lambda. 
\end{align*}
The first line follows from our fiberwise computation, the second line from proper base change, the third line from the fact that $(\id_{\Ga} - \tau)\tilde{i}_s$ maps to $0 \in \Ga$ and the last one follows from the Kunneth formula and $\RGamma_c(\Ga, \Lambda[2])= \Lambda$. 
\end{proof}

\subsection*{Identification of the invariants.}

We now prove lemma \ref{lem2}. 

\begin{lem}\label{lemAction1}
Let $H$ be a connected unipotent algebraic group, let $X = Y \times H$ be a $H$-scheme with $H$ acting trivially on Y, let $K$ be a sheaf on $X \times X$ equivariant for the diagonal action of $H$, and let $\psi$ be character of $H^{\Frob_H}$. Then there is a $H^{\Frob_H}$-equivariant isomorphism. 
\begin{equation*}
\RGamma_c(X,(\id \times \Frob_X)^* \Av^1_{\psi} K) \simeq e_{\psi} \RGamma_c(X,(\id \times \Frob_X)^*K).
\end{equation*}
\end{lem}

\begin{proof}
We begin with the following diagram,
\[\begin{tikzcd}
	& \pt \\
	{H \times Y \times H} & {Y \times H} \\
	{H \times Y \times H \times Y \times H} & {Y \times H \times Y \times H} \\
	{Y \times H \times Y \times H} & {Y \times Y \times H}
	\arrow["q"', from=2-2, to=1-2]
	\arrow["q", from=2-1, to=1-2]
	\arrow["{\pr_2}", from=2-1, to=2-2]
	\arrow["\alpha"', from=2-1, to=3-1]
	\arrow["{(\id \times \Frob_{Y \times H})}", from=2-2, to=3-2]
	\arrow["m"', from=3-1, to=3-2]
	\arrow["a", from=3-2, to=4-2]
	\arrow["{\pr_2}"', from=3-1, to=4-1]
	\arrow["a"', from=4-1, to=4-2]
\end{tikzcd}\]
where 
\begin{enumerate}
\item the map $m$ is the action map of $H$, that is $m(h_1, y_1, h_2, y_2, h_3) = (y_1, h_1h_2, y_2, h_3)$, 
\item the maps $\pr_2$ are the projections onto the last factors, 
\item the map $a$ is the quotient map by the diagonal action of $H$, that is $a(y_1,h_1,y_2,h_2) = (y_1, y_2, h^{-1}_1h_2)$, 
\item and $\alpha$ is the map $\alpha(h_1, y_1, h_2) = (h_1, y_1, h_1^{-1}h_2, \Frob_Y(y_1), \Frob_H(h_2))$. 
\end{enumerate}
The bottom square is not commutative and the top square is Cartesian. Let $K'$ be the sheaf on $Y \times Y \times H$ whose pullback to $X \times X$ is $K$. Denote also by $\zeta$ the following composition 
\begin{equation*}
H \times Y \times H \xrightarrow{\alpha} H \times Y \times H \times Y \times H \xrightarrow{\pr_2} Y \times H \times Y \times H \xrightarrow{a} Y \times Y \times H,
\end{equation*}
this map is given by $\zeta(h_1, y, h_2) = (y, \Frob_Y(y), h_2^{-1}h_1\Frob_H(h_2))$. Note that the following diagram is Cartesian 
\[\begin{tikzcd}
	{H \times Y \times H} & {Y \times H} \\
	{H \times Y \times H} & {Y \times Y \times H}
	\arrow["\zeta"', from=2-1, to=2-2]
	\arrow["{\id_Y \times \Frob_Y \times \mathcal{L}_H}", from=1-2, to=2-2]
	\arrow["{\tilde{m}}", from=1-1, to=1-2]
	\arrow["{\mathcal{L} \times\id_Y \times \id_H}"', from=1-1, to=2-1]
\end{tikzcd}\]
where $\tilde{m}(h_1,y,h_2) = (y, \Frob_Y(y), h_1h_2)$ and $\mathcal{L}_H$ is the Lang map of $H$. Finally note that $e_{\psi}\mathcal{L}_{H,*}\Lambda = \mathcal{L}_{\psi}$. Denote by $d_H = \dim H$, we can now compute

\begin{align*}
\RGamma_c(X, (\id \times \Frob_X)^*\Av^1_{\psi}K) &= q_!(\id \times \Frob_X)^*m_!(\mathcal{L}_{\psi}[2d_H] \boxtimes a^*K') \\
&= q_!\pr_{2,!}\alpha^*(\mathcal{L}_{\psi}[2d_H] \boxtimes a^*K') \\
&= q_!\pr_{2,!}\alpha^*(\pr_H^*\mathcal{L}_{\psi}[2d_H] \otimes \pr_2^*a^*K') \\
&= q_!\pr_{2,!}(\alpha^*\pr_H^*\mathcal{L}_{\psi}[2d_H] \otimes \alpha^*\pr_2^*a^*K') \\
&= q_!\pr_{2,!}(\pr_H^*\mathcal{L}_{\psi}[2d_H] \otimes \zeta^*K') \\
&= e_{\psi}q_!\pr_{2,!}(\pr_H^*\mathcal{L}_{H,!}\Lambda[2d_H] \otimes \zeta^*K') \\
&= e_{\psi}q_!\pr_{2,!}(\mathcal{L}_{H} \times \id_Y \times \id_H)_!\pr_H^*\Lambda[2d_H] \otimes \zeta^*K') \\
&= e_{\psi}q_!\pr_{2,!}(\mathcal{L}_{H} \times \id_Y \times \id_H)_!(\pr_H^*\Lambda[2d_H] \otimes (\mathcal{L}_{H} \times \id_Y \times \id_H)^*\zeta^*K') \\
&= e_{\psi}q_!\pr_{2,!}(\mathcal{L}_{H} \times \id_Y \times \id_H)_!(\mathcal{L}_{H} \times \id_Y \times \id_H)^*\zeta^*K'[2d_H] \\
&= e_{\psi}q_!\pr_{2,!}(\mathcal{L}_{H} \times \id_Y \times \id_H)_!\tilde{m}^*(\id_Y \times \Frob_Y \times \mathcal{L}_H)^*K'[2d_H] \\
&= e_{\psi}q_!(\id_Y \times \Frob_Y \times \mathcal{L}_H)^*K' \\
&= e_{\psi}\RGamma_c(X, (\id \times \Frob_X)^*K).
\end{align*}
The first line is obtained from proper base change, the second one by the projection formula. We then unfold the $\boxtimes$ and denote by $\pr_H$ the projection onto the first copy of $H$ from both $H \times X \times X$ and $H \times X$, in particular we have $\pr_H = \pr_H \alpha$ which yields the fifth line. The sixth line come from the identification $\mathcal{L}_{\psi} = e_{\psi}\mathcal{L}_H \Lambda$. The seventh line comes from proper base change via $\mathcal{L}_H \pr_H = \pr_H (\mathcal{L}_{H} \times Y \times H)$. The eighth line comes from the projection formula. The tenth line from the second diagram above and the eleventh one from the Kunneth formula and the fact that $\RGamma_c(H, \Lambda) = \Lambda[-2d_H]$. The final line comes from the fact that $(\id_Y \times \Frob_Y \times \mathcal{L}_H) = (\id \times \Frob_X)^*a^*$. 
\end{proof}

\begin{proof}[Proof of lemma \ref{lem2}]
From lemma \ref{lem1}, we know that $\Av_{\psi}K_w$ is supported on $\bar{U}TU/U \times \bar{U}TU/U$ which is a space of the form $X \times \bar{U}$ so lemma \ref{lemAction1} applies, giving an isomorphism $\RGamma_c(\bar{U}TU/U, (\id \times \Frob)^*\Av_{\psi}K_w) = e_{\psi}\RGamma_c(\bar{U}TU/U, (\id \times \Frob)^*K_w)$. 

Now the space $G/U$ is stratified by $\bar{B}$-orbits, given $\bar{B}wU/U$ a $\bar{B}$-orbit and $A$ a $\bar{U}^{\Frob_{\bar{U}}}$-equivariant sheaf on $\bar{B}wU/U$, we have $e_{\psi}\RGamma_c(\bar{B}wU, A) = 0$ if $w \neq 1$, this follows immediately from the genericity of $\psi$. Therefore $e_{\psi}\RGamma_c(G/U, (\id \times \Frob)^*K_w) = e_{\psi}\RGamma_c(\bar{U}TU/U, (\id \times \Frob)^*K_w)$. 
\end{proof}

\bibliographystyle{alpha}
\bibliography{KLSheaves}

\newcommand{\etalchar}[1]{$^{#1}$}
\begin{thebibliography}{AGK{\etalchar{+}}21}

\bibitem[AGK{\etalchar{+}}21]{6Authors}
D.~Arinkin, D.~Gaitsgory, D.~Kazhdan, S.~Raskin, N.~Rozenblyum, and
  Y.~Varshavsky.
\newblock Automorphic functions as the trace of frobenius.
\newblock 2021.

\bibitem[BBD82]{BBD}
A.~A. Beilinson, J.~Bernstein, and Pierre Deligne.
\newblock {\em Analyse et topologie sur les espaces singuliers {(I)}}.
\newblock Number 100 in Ast\'erisque. Soci\'et\'e math\'ematique de France,
  1982.

\bibitem[BBM04]{BBM}
Roman Bezrukavnikov, Alexander Braverman, and Ivan Mirkovic.
\newblock Some results about geometric whittaker model.
\newblock {\em Advances in Mathematics}, 186(1):143--152, 2004.

\bibitem[BDon]{BezDesh}
Roman Bezrukavnikov and Tamnay Deshpande.
\newblock On the geometric whittaker model.
\newblock In preparation.

\bibitem[BP98]{BravermanPolishchuk}
Alexander Braverman and Alexander Polishchuk.
\newblock Kazhdan-laumon representations of finite chevalley groups, character
  sheaves and some generalization of the lefschetz-verdier trace formula, 1998.

\bibitem[BR09]{BonnafeRouquier}
C\'edric Bonnaf\'e and Rapha\"el Rouquier.
\newblock Compactification des vari\'et\'es de {Deligne-Lusztig}.
\newblock {\em Annales de l'Institut Fourier}, 59(2):621--640, 2009.

\bibitem[DL76]{DeligneLusztig}
P.~Deligne and G.~Lusztig.
\newblock Representations of reductive groups over finite fields.
\newblock {\em Annals of Mathematics}, 103(1):103--161, 1976.

\bibitem[Dud09]{Dudas}
Olivier Dudas.
\newblock Deligne-lusztig restriction of a {Gelfand-Graev} module.
\newblock {\em Annales scientifiques de l'\'Ecole Normale Sup\'erieure}, Ser.
  4, 42(4):653--674, 2009.

\bibitem[GV20]{GaiVar}
Dennis Gaitsgory and Yakov Varshavsky.
\newblock Local terms for the categorical trace of frobenius.
\newblock 2020.

\bibitem[Kat88]{Katz}
Nicholas~M. Katz.
\newblock {\em Gauss Sums, Kloosterman Sums, and Monodromy Groups. (AM-116)}.
\newblock Princeton University Press, 1988.

\bibitem[Kaz95]{Kazhdan1995}
David Kazhdan.
\newblock {\em ``Forms'' of the Principal Series for GLn}, pages 153--171.
\newblock Birkh{\"a}user Boston, Boston, MA, 1995.

\bibitem[KL88]{KazhdanLaumon}
D.~Kazhdan and G.~Laumon.
\newblock Gluing of perverse sheaves and discrete series representation.
\newblock {\em Journal of Geometry and Physics}, 5(1):63--120, 1988.

\bibitem[KW01]{KW}
R.~Khiel and R.~Weissauer.
\newblock {\em Weil conjectures, Perverse sheaves and $\ell$-adic Fourier
  transform}.
\newblock Ergebnisse der Mathematik und ihrer Grenzgebiete. 3. Folge / A Series
  of Modern Surveys in Mathematics. Springer Berlin, Heidelberg, 2001.

\bibitem[Lau87]{LaumonFourier}
G\'erard Laumon.
\newblock Transformation de {Fourier,} constantes d'\'equations fonctionnelles
  et conjecture de {Weil}.
\newblock {\em Publications Math\'ematiques de l'IH\'ES}, 65:131--210, 1987.

\bibitem[Li21]{TZUJAN}
Tzu-Jan Li.
\newblock On endomorphism algebras of gelfand-graev representations.
\newblock 2021.

\end{thebibliography}

\end{document}